\documentclass[leqno,11pt,a4paper]{article}
\usepackage[left=2.0cm, top=3.0cm,bottom=3.5cm,right=2.0cm]{geometry} 
\usepackage{amsmath,amssymb,amsthm,mathrsfs,calc,graphicx}
\usepackage[british]{babel}
\usepackage{eufrak}

\numberwithin{equation}{section}

\newtheorem {theorem}{Theorem}[section]
\newtheorem {proposition}[theorem]{Proposition}
\newtheorem {lemma}[theorem]{Lemma}
\newtheorem {corollary}[theorem]{Corollary}

{\theoremstyle{definition}

}
{\theoremstyle{theorem}
\newtheorem {remark}[theorem]{Remark}

}

\def\ba{\begin{array}}
\def\ea{\end{array}}
\def\bea{\begin{eqnarray} \label}
\def\eea{\end{eqnarray}}
\def\be{\begin{equation} \label}
\def\ee{\end{equation}}
\def\bit{\begin{itemize}}
\def\eit{\end{itemize}}
\def\ben{\begin{enumerate}}
\def\een{\end{enumerate}}

\def\lan{\langle}
\def\ran{\rangle}

\def\DD{\mathbb{D}}
\def\EE{\mathbb{E}}

\def\NN{\mathbb{N}}
\def\PP{\mathbb{P}}

\def\RR{\mathbb{R}}

\def\VV{\mathbb{V}}

\def\ZZ{\mathbb{Z}}

\def\bF{\mathbf{F}}
\def\bG{\mathbf{G}}

\def\bX{\mathbf{X}}
\def\bY{\mathbf{Y}}

\def\EH{\EuFrak{H}}
\def\cA{\mathcal{A}}

\def\cF{\mathcal{F}}

\def\cH{\mathcal{H}}

\def\cL{\mathcal{L}}
\def\cN{\mathcal{N}}

\def\cS{\mathcal{S}}

\def\dint{\textup{d}}

\allowdisplaybreaks

\begin{document}

\title{\bfseries Malliavin-Stein method for\\ Variance-Gamma approximation on Wiener space}

\author{Peter Eichelsbacher\footnotemark[1] \, and  Christoph Th\"ale\footnotemark[2]}

\date{}
\renewcommand{\thefootnote}{\fnsymbol{footnote}}
\footnotetext[1]{Ruhr University Bochum, Faculty of Mathematics, NA 3/66, D-44780 Bochum, Germany. E-mail: peter.eichelsbacher@rub.de}

\footnotetext[2]{Ruhr University Bochum, Faculty of Mathematics, NA 3/68, D-44780 Bochum, Germany. E-mail: christoph.thaele@rub.de}

\maketitle

\begin{abstract}
We combine Malliavin calculus with Stein's method to derive bounds for the Variance-Gamma approximation of functionals of isonormal Gaussian processes, in particular of random variables living inside a fixed Wiener chaos induced by such a process. The bounds are presented in terms of Malliavin operators and norms of contractions. We show that a sequence of distributions of random variables in the second Wiener chaos converges to a Variance-Gamma distribution if and only if their moments of order two to six converge to that of a Variance-Gamma distributed random variable (six moment theorem). Moreover, simplified versions for Laplace or symmetrized Gamma distributions are presented. Also multivariate extensions and a universality result for homogeneous sums are considered.

\bigskip

\end{abstract}

\section{Introduction}
Let $X = \{ X(h)\}_{h \in \EH }$ be an isonormal Gaussian process, defined on a probability space $(\Omega, \cF, \PP)$, over some real separable Hilbert space $\EH$, fix an integer $q \geq 2$ and let $\{F_n\}_{n\in\NN}$ be a sequence of random variables belonging to the $q$th Wiener chaos induced by $X$ (precise definitions follow in Section \ref{sec:Malliavin} below). Denote by $\EH^{\otimes q}$ and $\EH^{\odot q}$ the $q$th tensor product and the $q$th symmetric tensor product of $\EH$, respectively, and let $I_q$ be the isometry between $\EH^{\odot q}$ (equipped with the modified norm $\sqrt{q!}\, \| \,\cdot\, \|_{\EH^{\otimes q}}$) and the $q$th Wiener chaos of $X$. If $\EH$ in particular is an $L^2$-space of some $\sigma$-finite measure space without atoms, then a random variable $I_q(h)$ with $h \in \EH^{\odot q}$ has the form of a multiple Wiener-It\^o integral of order $q$. 

In recent years, many efforts have been made to characterize those sequences $\{F_n\}_{n\in\NN}$ belonging to a Wiener chaos of fixed order, which verify a central limit theorem in the sense that $F_n$ converges in distribution, as $n\to\infty$, to a centered Gaussian random variable $N$ of unit variance (compare with the book \cite{NP:book} for an overview). The celebrated fourth moment theorem
of Nualart and Peccati \cite{NualartPeccati:2005} asserts that if $\EE[F_n^2] =1$ for all $n \geq 1$, then, as $n \to \infty$, $F_n$ converges in distribution to $N$ if and only if $\EE[F_n^4]$ converges to $3$, the fourth moment of the Gaussian random variable $N$. This can be seen as a drastical simplification of the classical method of moments, which ensures convergence in distribution of $F_n$ to $N$, provided that \textit{all} moments of $F_n$ converge to that of $N$.

Only a few years later, Nourdin and Peccati \cite{NourdinPeccati:2009} combined Stein's method for normal approximation with the Malliavin calculus on the Wiener space of variations to prove explicit bounds on the total variation distance $d_{TV}(F_n,N) := \sup_{A \in \mathcal{B}(\RR)} |P(F_n \in A) - P(N \in A)|$, where the supremum runs over all bounded Borel sets $A \subset \RR$ (in fact, also other notions of distances have been considered in \cite{NourdinPeccati:2009}, but we restrict here to the total variation distance). They showed that for a sequence $\{F_n\}_{n\in\NN}$ of random variables of the form $F_n=I_q(h_n)$ with $h_n\in\EH^{\odot q}$ it holds that
\begin{equation}\label{eq:IntroRate}
d_{TV}(F_n,N) \leq 2 \sqrt{ \frac{q-1}{3q} ( \EE[F_n^4] -3)}=2 \sqrt{ \frac{q-1}{3q} \kappa_4(F_n)},
\end{equation}
where, for integers $j\geq 1$, we write $\kappa_j(X)$ for the $j$th cumulant (semi-invariant) of a random variable $X$. More recently, in \cite{NourdinPeccati:2013}, exact rates of convergence for the total variation distance have been found. Namely, if $F_n$ converges in distribution to $N$, as $n\to\infty$, then there exist two constants $0<c<C<\infty$ (possibly depending on $q$ and on $\{F_n\}$, but not on $n$), such that
\begin{equation}\label{eq:IntroRateOp}
c M(F_n) \leq d_{TV}(F_n,N) \leq C M(F_n)\quad\text{with}\quad
M(F_n) := \max \{ |\kappa_3(F_n)|,|\kappa_4(F_n)| \}
\end{equation}
for all $n\in\NN$, where, recall, $\kappa_3(F_n)=\EE[F_n^3]$ and $\kappa_4(F_n)=\EE[F_n^4]-3(\EE[F_n^2])^2=\EE[F_n^4] -3$ are the third and the fourth cumulant of $F_n$, respectively. In other words this means that the rate provided by \eqref{eq:IntroRate} is suboptimal by a squareroot factor. This, however, seems unavoidable using the Malliavin-Stein technique for normal approximation, which is based on the analysis of fourth moments, while the proof of \eqref{eq:IntroRateOp} uses more refined arguments.

The main goal of this paper is to study non-central limit theorems (i.e., limit theorems with a non-Gaussian limiting distribution) for sequences $\{F_n\}_{n\in\NN}$ belonging to a fixed Wiener chaos of order $q\geq 2$, as above. A first step in this direction is the paper \cite{NourdinPeccati:2009b} by Nourdin and Peccati, in which conditions on the sequence $\{F_n\}_{n\in\NN}$ have been derived, under which convergence towards a centred Gamma distribution takes place. Moreover, in \cite{NourdinPeccati:2009} rates of convergences for such Gamma approximations were considered, again by applying Stein's method. It is an interesting fact that if $q$ is an odd integer, there is no sequence of chaotic random variables with bounded variances converging in distribution to a centred Gamma distribution. This is a consequence of the fact that a random variable $I_q(h)$ with $h\in\EH^{\odot q}$ with $q$ being odd has third moment equal to zero, while the third moment of a centred Gamma distribution is strictly positive. Beyond the 
normal and Gamma approximation results in \cite{NourdinPeccati:2009}, up to our best knowledge there are no other quantitative limit theorems for chaotic sequences so far. Our paper is an attempt to fill this gap in case of the broad class of so-called Variance-Gamma distributions. This is a $3$-parametric family of continuous probability distributions on the real line defined as variance-mean mixtures of Gaussian random variables with a Gamma mixing distribution. We emphasize that Variance-Gamma distributions are widely used in financal mathematics, for example. Particular examples of Variance-Gamma distributions are the Laplace distribution or more general symmetrized Gamma distributions, but also the classical normal or Gamma distribution, which show up as limiting cases. It is interesting to see that in our set-up, no parity condition on $q$ is necessary in general. It is also worth mentioning the recent work \cite{KusuokaTudor:2013} of Kusuoka and Tudor, where it has been shown that within the so-called 
Pearson-family of probability distributions on the real line, only the normal and the Gamma distribution can appear as limit laws for a sequence $\{F_n\}_{n\in\NN}$ of chaotic random variables. The Laplace distribution, the symmetrized Gamma distribution and, more generally, the Variance-Gamma distributions, are not members of the Pearson-family and this way our study goes beyond the theory developed in \cite{KusuokaTudor:2013}.

Besides the goal of identifying new limiting distributions for the sequence $\{F_n\}_{n\in\NN}$ as considered above, another source of motivation for our paper comes from the area of free probability. In \cite{DeyaNourdin:2012}, Deya and Nourdin studied the convergence of a sequence of multiple stochastic integrals with respect to a free Brownian motion to what they call the tetilla law, which can be regarded as the commutator of the well-known Marchenko-Pastur distribution. Our aim here is to identitfy the non-free analogue of this distribution and to prove a related limit theorem for multiple stochastic integrals with respect to the classical Brownian motion. We will see that the Laplace distribution with parameter $\sqrt{2}$, which, as already mentioned above, is contained in the class of Variance-Gamma distributions, can be seen as such an analogue, see Remark \ref{rem:tetilla} below for a more detailed description. 

\medskip

Let us describe our results and the structure of our paper in some more detail. In Section \ref{sec:Malliavin} we collect some background material related to isonormal Gaussian processes and the Malliavin calculus of variations on the Wiener space. We will recall in particular the definitions of the so-called  $\Gamma$-operators, which are central for our further investigations. Essential elements of Stein's method for Variance-Gamma distributions are reviewed in Section \ref{subsec:Stein}. In particular, we state bounds on the solution of the Stein equation and introduce some particular subclasses and limiting cases of Variance-Gamma distributions, which are of special interest. An abstract bound for the Wasserstein distance $d_W(F,Y)$ between a (suitably regular) functional $F$ of an isonormal Gaussian process and a Variance-Gamma distributed random variable $Y$ in the spirit of the Malliavin-Stein method is derived in Section \ref{sec:GeneralBound}. We will see that the bound is expressed in terms of the $\Gamma$-operators 
mentioned above. This general bound is specialized in Section \ref{subsec:Generalqgeq2} to the case of elements living inside a fixed Wiener chaos of order $q\geq 2$. We derive a sufficient analytic criterium in terms of contractions for such a sequence to converge in distribution to a Variance-Gamma distributed random variable. In this context, we also recover the fourth moment theorem discussed above together with a rate of convergence for the Wasserstein distance, which improves \eqref{eq:IntroRate}, but is still not optimal in view of \eqref{eq:IntroRateOp}. Our general bound is specialized in Section \ref{double} to the case $q=2$, which is of particular interest in view of the theory of quadratic forms, for example. We show that in this case the previously derived sufficient criterium for convergence turns out to be necessary and, moreover, equivalent to a simple moment condition involving moments of up to order six only. For example, we 
show that a sequence $\{F_n\}_{n\in\NN}$ of elements belonging to a second Wiener chaos converges in distribution to a random variable $Y$ having a Laplace distribution with parameter $b>0$ if and only if 
\begin{equation}\label{eq:IntroMoments}
\EE[F_n^2]\to 2b^2,\quad\EE[F_n^4]\to 4! b^4\quad\text{and}\quad
\EE[F_n^6] \to 6! b^6,
\end{equation}
as $n\to\infty$. This this case, we also have the bound 
\begin{equation}\label{eq:IntoBound}
\begin{split}
d_W ( F_n, Y) &\leq C_1
\Big( \frac{1}{120} \kappa_6(F_n) - \frac{1}{6} \kappa_4(F_n) \kappa_2(F_n) + \frac{1}{4} \kappa_2(F_n)^3 +  \frac 16 \kappa_3(F_n)^2 \Big)^{1/2}\\
&\hspace{6cm}+C_2 \big| 2b^2 - \kappa_2(F_n) \big|
\end{split}
\end{equation}
on the Wasserstein distance, where, recall, $\kappa_j(F_n)$ stands for the $j$th cumulant of $F_n$ and where $C_1,C_2>0$ are constants only depending on the parameter $b$. We like to emphasize the following interesting observation. Namely, although the third cumulant $\kappa_3(F_n)$ shows up in the bound \eqref{eq:IntoBound}, it automatically vanishes in the limit, as $n\to\infty$, under the moment condition \eqref{eq:IntroMoments}. Our result can be seen as a six moment theorem for the convergence to a Laplace distribution on the second Wiener chaos. An analogue for general Variance-Gamma distributions is one of the main achievements of this paper. We mention that this is also closely connected to the work \cite{NourdinPoly:2012} (see also the erratum \cite{NourdinPoly:2012ERR}) of Nourdin and Poly, who characterize convergence of a sequences of random elements inside the second Wiener chaos associated with the ordinary (and the free) Brownian motion in terms of conditions on a sequence of consequtive moments. However, their results do not allow to derive rates of convergence. In the final Section \ref{deJong} we deal with a universality question for so-called homogeneous sums with respect to Variance-Gamma convergence as well as with some multivariate extensions of the previously derived results.

\medskip
The results of our paper complement those obtained in the recent study of Azmoodeh, Peccati and Poly \cite{APP:2014}, which has independently been conducted in parallel with our paper. They derive necessary and sufficient conditions under which a sequence $\{F_n\}_{n\in\NN}$ as above converges to a limiting random variable, whose distribution is a finite linear combination of centred $\chi^2$-distributions. However, these limit theorems are not quantitative in the sense that they just state the convergence in distribution without giving upper bounds on the rate of convergence. On the other hand, the results are for sequences living inside a Wiener chaos of arbitrary order.

\section{Elements of Gaussian analysis and Malliavin calculus}\label{sec:Malliavin}

\paragraph{Isonormal Gaussian processes.}

Here we collect the essentials of Gaussian analysis and Malliavin calculus that are used in the paper, see the books \cite{Nualart:2006} and
\cite{NP:book} for further details.

For a real separable Hilbert space $\EH$ and $q \geq 1$, we write $\EH^{\otimes q}$ and $\EH^{\odot q}$ to indicate, respectively, the $q$th tensor power and the $q$th symmetric tensor power of $\EH$ with convention $\EH^{\otimes 0} = \EH^{\odot 0} =\RR$. We denote by $X = \{ X(h) \}_{h \in \EH }$ an isonormal Gaussian process over $\EH$, i.e., $X$ is a centred Gaussian family, defined on some probability space $(\Omega, \cF, \PP)$ and indexed by $\EH$, with a covariance structure given by the relation $\EE [X(h) X(g) ] = \langle h, g \rangle_{\EH}$. We assume that $\cF = \sigma(X)$. For $q \geq 1$, the symbol $\cH_q$ denotes the $q$th Wiener chaos of $X$, defined as the closed linear subspace of $L^2(\Omega, \cF, \PP)=:L^2(\Omega)$, which is generated by the family $\{ H_q(X(h)) : h \in \EH, \|h\|_{\EH} =1 \}$, where $H_q(x) =(-1)^q e^{x^2/2} \frac{d^q}{dx^q} (e^{-x^2/2})$ is the $q$th Hermite polynomial. For any $q \geq 1$ the mapping $I_q(h^{\otimes q}) = H_q(X(h))$ can be extended to a linear isometry 
between $\EH^{\odot q}$ and the $q$th Wiener chaos $\cH_q$.
For $q=0$ we write $I_0(c)=c$, $c \in \RR$. When $\EH = L^2(A, \cA, \mu) =: L^2(\mu)$ with $\mu$ being a non-atomic $\sigma$-finite measure on the measurable space $(A, \cA)$, for every $f \in \EH^{\odot q} =L_s^2(\mu^q)$ the random variable $I_q(f)$ coincides with the $q$-fold multiple Wiener-It\^o-integral of $f$ with respect to the centred Gaussian measure canonically generated by $X$, see \cite[Section 1.1.2]{Nualart:2006}. Here, $L_s^2(\mu^q)$ stands for the subspace of $L^2(\mu^q)$ composed by symmetric functions. It is well-known that $L^2(\Omega)$ can be decomposed into the infinite orthogonal sum of the spaces $\cH_q$. Hence an $F \in L^2(\Omega)$ admits the Wiener-It\^o chaotic expansion
\begin{equation} \label{chaos}
F = \sum_{q=0}^{\infty} I_q(f_q),
\end{equation}
with $f_0 = \EE[F]$, and the $f_q \in \EH^{\odot q}, q \geq 1$, uniquely determined by $F$.

Let $\{ e_n\}_{n\in\NN}$ be a complete orthonormal system in $\EH$. For $f \in \EH^{\odot p}$ and  $g \in \EH^{\odot q}$, for every $r =0, \ldots, p \wedge q$, the contraction of $f$ and $g$ of order $r$ is the element of $\EH^{\otimes (p+q-2r)}$ defined by
$$
f \otimes_r g = \sum_{i_1, \ldots, i_r =1}^{\infty} \langle f, e_{i_1} \otimes \cdots \otimes e_{i_r} \rangle_{\EH^{\otimes r}} \otimes 
\langle g, e_{i_1} \otimes \cdots \otimes e_{i_r} \rangle_{\EH^{\otimes r}}.
$$ 
It is important to notice that the definition of $f \otimes_r g$ does not depend on the particular choice of $\{ e_n\}_{n\in\NN}$, and that $f \otimes_r g$ is not necessarily symmetric. We denote its canonical symmetrization by $f \widetilde{\otimes}_r g \in \EH^{\odot (p+q-2r)}$. Clearly, $f \otimes_0 g = f \otimes g$ and
$f \otimes_q g = \langle f, g \rangle_{\EH^{\otimes q}}$. Moreover, when $\EH = L^2(\mu)$ and $r =1, \ldots, p \wedge q$, the contraction $f \otimes_r g$ is the element of $L^2(\mu^{p+q-2r})$ given by
\begin{eqnarray*}
f  \otimes_r g(x_1, \ldots, x_{p+q-2r})  =  \\
&& \hspace{-3cm} \int_{A^r} f(x_1, \ldots, x_{p-r}, a_1, \ldots, a_r) \, g(x_{p-r+1}, \ldots, x_{p+q-2r}, a_1, \ldots, a_r) d(\mu^r(a_1,\ldots,a_r)).
\end{eqnarray*} 
We will intensively use the isometry property and the product formula for multiple integrals, i.e.\ elements of a fixed Wiener chaos. Namely, if $f \in \EH^{\odot p}$ and $g \in \EH^{\odot q}$, and $1 \leq q \leq p$, then
\begin{equation} \label{isometry}
\EE \bigl[ I_p(f) I_q(g) \bigr] =  p! \langle f, g \rangle_{\EH^{\otimes p}}\,{\bf 1}(p=q),
\end{equation}
and
\begin{equation} \label{productrule}
I_p(f) \, I_q(g) = \sum_{r=0}^{p \wedge q} r! {p \choose r} {q \choose r} I_{p+q-2r} ( f \widetilde{\otimes}_r g ),
\end{equation}
see \cite[Proposition 1.1.3]{Nualart:2006}.

\paragraph{Malliavin operators.}
Let $X$ be an isonormal Gaussian process and let $\cS$ be the set of random variables of the form $F= g(X(\phi_1), \ldots, X(\phi_n))$
with $n \geq 1$, $\phi_1,\ldots,\phi_n\in\EH$ and $g: \RR^n \to \RR$ an infinitely differentiable function whose partial derivatives have polynomial growth. The Malliavin derivative of $F$ with respect to $X$ is the element of $L^2(\Omega, \EH)$ defined as
$$
D F = \sum_{i=1}^n \frac{\partial g}{\partial x_i} (X(\phi_1), \ldots, X(\phi_n)) \phi_i.
$$
Hence $D X(h) =h$ for $h \in \EH$. By iteration, the $m$th derivative $D^m F$ is an element of $L^2(\Omega, \EH^{\odot m})$
for every $m \geq 2$. For $m \geq 1$ and $p\geq 1$, $\DD^{m,p}$ denotes the closure of $\cS$ with respect to the norm
$$
\|F\|_{m,p}^p = \EE [|F|^p] + \sum_{i=1}^m \EE \bigl[ \| D^i F\|_{\EH^{\otimes i}}^p \bigr].
$$
We use the notation $\DD^{\infty} := \bigcap_{m\geq 1} \bigcap_{p \geq 1} \DD^{m,p}$. Every finite linear combination of
multiple Wiener-It\^o integrals is an element of $\DD^{\infty}$ and its law admits a density with respect to the Lebesgue measure on the real line.
The Malliavin derivative satisfies the following chain rule. If $\varphi : \RR^n \to \RR$ is continuously differentiable with bounded partial derivatives and if $F=(F_1, \ldots, F_n)$ is a vector of elements of $\DD^{1,2}$, then $\varphi(F) \in \DD^{1,2}$ and
\begin{equation} \label{chainrule}
D  \varphi(F) = \sum_{i=1}^n \frac{\partial \varphi}{\partial x_i}(F) \, DF_i.
\end{equation}
If $\EH = L^2(A, \cA, \mu)$ with $\mu$ $\sigma$-finite and non-atomic, then the derivative of $F$ as in \eqref{chaos} is given by
\begin{equation} \label{diffchaos}
D_x F = \sum_{q=1}^{\infty} q \, I_{q-1} (f_q(\cdot,x)), \quad x \in A,
\end{equation}
where $f_q(\cdot,x)$ stands for the function $f_q$ with one of its arguments fixed to be $x$.
The adjoint of the operator $D$ is denoted by $\delta$ and called the divergence operator. A random element $u \in L^2(\Omega, \EH)$ belongs to the domain of $\delta$ ($\text{Dom}(\delta)$), if and only if it verifies $| \EE[ \langle DF, u \rangle_{\EH}]| \leq c_u \|F\|_{L^2(\Omega)}$ for any $F \in \DD^{1,2}$, where $c_u$ is a constant depending only on $u$. For $u \in \text{Dom}(\delta)$
the random variable $\delta(u)$ is defined by the integration-by parts formula
\begin{equation} \label{intby}
\EE [F \delta(u) ] = \EE [ \langle DF, u \rangle_{\EH} ],
\end{equation}
which holds for every $F \in \DD^{1,2}$. The infinitesimal generator of the Ornstein-Uhlenbeck semi-group is given by
$L = \sum_{q=0}^{\infty} - q \, J_q$, where $J_q(F) :=I_q(f_q)$ for every $F$ as in \eqref{chaos}. The domain of $L$ is $\DD^{2,2}$.
A random variable $F$ belongs to $\DD^{2,2}$ if and only if $F \in \text{Dom}(\delta D)$ (i.e., $F \in \DD^{1,2}$ and $DF \in \text{Dom}(\delta)$) and in this case,
\begin{equation} \label{rel}
\delta D F = - L F.
\end{equation}
For any $F \in L^2(\Omega)$ we define $L^{-1} F = \sum_{q=1}^{\infty} - \frac 1q J_q(F)$. The operator $L^{-1}$ is called the pseudo-inverse of $L$. For any $F \in L^2(\Omega)$ one has that $L^{-1}F \in \text{Dom} L = \DD^{2,2}$, and
\begin{equation} \label{inverse}
L L^{-1} F = F - \EE[F].
\end{equation}
The following result is used frequently throughout this paper (see \cite[Lemma 2.3]{BBNP:2012} and \cite[Lemma 2.1]{NourdinPeccati:2009b}).

\begin{lemma} \label{keyid}
(1) 
Suppose that $H \in \DD^{1,2}$ and $G \in L^2(\Omega)$. Then, $L^{-1} G \in \DD^{2,2}$ and
$$
\EE [H G] = \EE[H] \EE[G] + \EE [ \langle DH, -D L^{-1} G \rangle_{\EH} ].
$$
(2)
Suppose that $F=I_q(f)$ with $q \geq2$ and $f \in \EH^{\odot q}$. Then for every $s \geq 0$, we have
$$
\EE [F^s \|DF\|_{\EH}^2] = \frac{q}{s+1} \EE[F^{s+2}].
$$
\end{lemma}

\begin{proof}
(1) By \eqref{rel} and \eqref{inverse} we observe that
$$
\EE [H G] -\EE[H] \EE[G] = \EE [H(G - \EE[G])] = \EE[H \, L L^{-1} G] = \EE [H \delta(-DL^{-1}G)].
$$
The result is obtained by using the integration-by-parts formula \eqref{intby}.

(2) By chain rule \eqref{chainrule}, the integration-by-parts formula \eqref{intby} and \eqref{rel} we obtain
$$
\EE [F^s \|DF\|_{\EH}^2] = \frac{1}{s+1} \EE [ \langle DF, D(F^{s+1}) \rangle_{\EH}] =  \frac{1}{s+1} \EE[\delta DF \times F^{s+1}] = \frac{q}{s+1} \EE[F^{s+2}],
$$
where we used that $-L I_q = q I_q= q F$.
\end{proof}

\paragraph{Cumulants and $\Gamma$-operators.}

Let $F$ be a real-valued random variable such that $\EE[|F|]^m < \infty$ for some integer $m \geq 1$ and define $\phi_F(t) = \EE[e^{itF}]$, $t \in \RR$, to be the characteristic function of $F$. Then, for $j=1, \ldots, m$, the $j$th cumulant of $F$, denoted by
$\kappa_j(F)$, is given by 
$$
\kappa_j(F) = (-i)^j \frac{d^j}{dt^j} \log \phi_F(t) \Big|_{t=0}.
$$
There is a well-known relation between cumulants and moments. In this paper, such a relation is needed for cumulants and moments up to order six, and only if $\EE[F]=0$. In this case, we have $\kappa_2(F) = \EE[F^2]$,  $\kappa_3(F) =\EE[F^3]$, $\kappa_4(F) = \EE[F^4] - 3 \EE[F^2]^2$ and $\kappa_6(F) = \EE[F^6] - 15 \EE[F^4] \EE[F^2] - 10 (\EE[F^3])^2 + 30 (\EE[F^2])^3$.

The cumulants can be characterized in terms of Malliavin operators. For this, we need to introduce the so-called $\Gamma$-operators $\Gamma_j$, $j\geq 1$. For $F \in \DD^{\infty}$ we define $\Gamma_1(F) = F$ and, for very $j \geq 2$,
\begin{equation} \label{gammadef}
\Gamma_j(F) = \langle DF, -DL^{-1} \Gamma_{j-1}(F) \rangle_{\EH}.
\end{equation}
Each $\Gamma_j(F)$ is well-defined and an element of $\DD^{\infty}$, since $F$ is assumed to be in $\DD^{\infty}$, see \cite[Lemma 4.2]{NourdinPeccati:2010}. According to \cite[Theorem 4.3]{NourdinPeccati:2010}, there is an explicit relation between $\Gamma_j(F)$ and the $j$th cumulant of $F$. Namely, if $F \in \DD^{\infty}$, then $F$ has finite moments of all orders and for each integer $j \geq 1$ it holds that
\begin{equation}\label{eq:KappaUndGamma}
 \kappa_j(F) = (j-1)! \EE [\Gamma_j(F)].
\end{equation}
The relation continuous to hold under weaker assumptions on the regularity of $F$, see \cite[Theorem 4.3]{NourdinPeccati:2010}. For $F \in \DD^{1,2}$, it follows that $\Gamma_2(F) \in L^1(\Omega)$ and $\VV[F]= \EE [\Gamma_2(F)]$ and for $F \in \DD^{1,4}$, it holds that $\Gamma_2(F) \in L^2(\Omega)$. 

If $F$ belongs to a fixed Wiener chaos (i.e, if $F$ has the form of a multiple integral if $\EH=L^2(\mu)$ as discussed above), there is a more explicit representation for $\Gamma_j(F)$, see Formula (5.25) in \cite{NourdinPeccati:2010}. To state it, let $q \geq 2$ and $F = I_q(f)$ with $f \in \EH^{\odot q}$. Then for any $j \geq 1$, applyling the product formula \eqref{productrule}, we have
 \begin{eqnarray} \label{gammarep}
 \Gamma_{j+1}(F) & = & \sum_{r_1=1}^q \cdots \sum_{r_j =1}^{[jq-2r_1- \ldots - 2 r_{j-1}] \wedge q} c_q(r_1, \ldots, r_j) 
 {\bf 1}_{ \{r_1 < q\} } \ldots {\bf 1}_{ \{r_1+ \ldots + r_{j-1} <  \frac{jq}{2} \} }  \\ \nonumber
 & & \qquad \times I_{(j+1)q - 2 r_1 - \ldots - 2 r_j} \bigl( ( \ldots (f \widetilde{\otimes}_{r_1} f)  \widetilde{\otimes}_{r_2} f) \ldots f)  \widetilde{\otimes}_{r_j} f \bigr),
 \end{eqnarray}
where the constants $c_q(r_1, \ldots, r_j)$ are recursively defined as follows:
 $$
 c_q(r) = q(r-1)! {q-1 \choose r-1}^2
 $$
 and for $a \geq 2$,
 $$
 c_q(r_1, \ldots, r_a) = q (r_a-1)! {aq -2 r_1 - \ldots - 2 r_{a-1} -1 \choose r_a -1} {q-1 \choose r_a -1} c_q(r_1, \ldots, r_{a-1}).
 $$

\section{Elements of Stein's method}\label{subsec:Stein}


\paragraph{Wasserstein distance and the standard normal distribution.}

 Stein's method is a set of techniques allowing to evaluate distances between probability measures. In the present paper, we focus on the Wasserstein distance ($L^1$-distance). For any two real-valued random variables $X$ and $Y$ it is defined as
 $$
 d_W(X,Y) : = \sup_{h \in \cL} \{ | \EE[h(X)] - \EE[h(Y)] | \}
 $$
 with $\cL := \{ h : \RR \to \RR: |h(x) -h(y)| \leq |x-y| \}$ (Lipschitz functions). We will make use of the fact that the elements in $\cL$ are exactly those absolutely continuous functions whose derivatives are a.e.\ bounded by $1$ in absolute value. We notice that $d_W(X_n,Y)\to 0$ as $n\to\infty$ for a sequence of random variables $\{X_n\}_{n\in\NN}$ implies convergence of $X_n$ to $Y$ in distribution (the converse is not necessarily true).
 
 A standard Gaussian random variable $Z$ is characterized by the fact that for every absolutely continuous function $f:\RR\to\RR$ for which $\EE\big[Zf(Z)\big]<\infty$ it holds that 
 \begin{equation} \label{steinnormal}
 \EE\big[f'(Z)-Zf(Z)\big]=0.
 \end{equation}
 This together with the definition of the Wasserstein distance is the motivation to study the Stein equation
\begin{equation}\label{eq:SteinEquation}
 f'(x)-xf(x) = h(x) - \EE[h(Z)],\quad x\in\RR.
\end{equation}
 A solution of the Stein equation is a function $f_h$, depending on $h$, which satisfies \eqref{eq:SteinEquation}. For $h \in \cL$, $f_h$ is bounded
 and twice differentiable such that $\|f_h'\|_{\infty} \leq 1$ and $\|f_h''\|_{\infty} \leq 2$, see \cite[Lemma 2.3]{Steinbuch2010}.
 If we replace $x$ by a random variable $F$ and take expectations in the Stein equation \eqref{eq:SteinEquation}, we infer that $$\EE\big[f_h'(F)-Ff_h(F)\big]=
\EE[h(F)] - \EE[h(Z)] $$ and hence $$d_W(F,Z) \leq \sup \{  | \EE[ f'(F)-F f(F)] |  : \|f'\|_{\infty} \leq 1 \, \text{and} \, \|f''\|_{\infty} \leq 2 \}.$$
With \eqref{steinnormal}, we obtain that for every $h \in \EH$ such that $\|h\|_{\EH}=1$ we have for smooth functions $f$ that $\EE\big[f'(X(h))-X(h)f(X(h))\big]=0$. It is a particular case of the consequence of Lemma \ref{keyid}(1), that for every $F \in \DD^{1,2}$ with mean zero, $\EE [ F f(F)] = 
\EE [ \langle DF, -D L^{-1} F \rangle_{\EH} f'(F)]$. Hence, by Cauchy-Schwarz, we obtain
$$
d_W(F,Z) \leq \EE \bigl[ ( 1 - \langle DF, -D L^{-1} F \rangle_{\EH} )^2 \bigr]^{1/2},
$$
(see \cite[Theorem 3.1]{NourdinPeccati:2009}).

\paragraph{Symmetric Gamma distributions.}

The main goal of our paper is to consider probabilistic approximations by Variance-Gamma random variables. To motivate the right choice of a Stein equation, first let us consider the case of Laplace distribution or, more generally, symmetrized Gamma distribution. The Lebesgue-density of a Laplace distribution with parameter $b$ is given by
\begin{equation} \label{laplacedensity}
p_b(x) = \frac{1}{2b} \exp \biggl( - \frac{|x|}{b} \biggr), \quad x \in \RR, b >0,
\end{equation}
while the Lebesgue-density of a symmetrized Gamma distribution with parameters $\lambda>0$ and $r>0$ equals
\begin{equation} \label{gammadensity}
p_{\lambda, r} (x) = \frac{\lambda^r}{ 2\Gamma(r)} |x|^{r-1} e^{-\lambda |x|}, \quad x \in \RR.
\end{equation}
In what follows we shall indicate the distribution with density $p_{\lambda,r}$ by $\Gamma_s(\lambda,r)$ and by $\Gamma(\lambda,r)$ we denote the non-symmetric (i.e., classical) Gamma distribution. Note that the choice $r=1$ and $\lambda = 1/b$ leads to the Laplace distribution with density as at \eqref{laplacedensity}. A first-order Stein operator for 
a random variable with density $p_b$ can be obtained by the so-called density approach, see \cite{DiaconisStein:2004}. In fact, if $Y$ has Lebesgue-density $p_b$, then $\EE[f'(Y) - \frac{p_b'(Y)}{p_b(Y)} f(Y) ] =0$ for all absolutely continuous $f$ for which the expectation exists. However, $\frac{p_b'(Y)}{p_b(Y)}= \text{sign}(Y)$ and it is in general technically highly sophisticated or even impossible to compute $\EE [ \text{sign}(Y) f(Y)]$. To overcome this difficulty, we put, if $f'$ is absolutely continuous, $G(x) = \text{sign}(x) (f(x)-f(0))$, to see that
$$
\EE [f''(Y)] = \frac 1b \EE[ \text{sign}(Y) f'(Y)] = \frac 1b \EE [G'(Y)] = \frac{1}{b^2} (\EE[f(Y)] -f(0)).
$$
Summarizing, we obtain that if $Y$ has a Laplace distribution with parameter $b$, $f$ and $f'$ are absolutely continuous functions and $\EE[f(Y)]$ exists, that
\begin{equation} \label{pikeren}
\EE[f(Y)] -f(0) = b^2 \EE[f''(Y)],
\end{equation}
see \cite[Lemma 1]{PikeRen:2012}. A major disadvantage of this characterization is that the machinery of Malliavin calculus usually enters by $\EE [ F f(F)] = \EE [ \langle DF, -D L^{-1} F \rangle_{\EH} f'(F)]$. Hence, we substitute $f(Y)$ by $Y f(Y)$ and obtain $\EE[Y f(Y)] = b^2 \EE[X f''(Y)+ 2 f'(Y)]$ if $Y$ is Laplace distributed with parameter $b$. It is interesting to see that this leads to the same Stein characterization of a Laplace distribution (and similarly of a $\Gamma_s(\lambda,r)$- distribution) as introduced recently in \cite{Gaunt:2013} from an entirely different perspective. We emphasize that second-order Stein operators are not commonly used in the literature, although in \cite{Pekoez:2013} the authors obtained a second-order Stein operator for the so-called 
Kummert-$U$ density.

\begin{lemma} \label{gaunt}
Let $Y$ be a real valued random variable. Then $Y$ is distributed according to the symmetrized Gamma distribution \eqref{gammadensity} with parameters $r$ and $\lambda$
if and only if, for all $f : \RR \to \RR$ such that $f$ is piecewise twice continuously differentiable and $\EE[|Y f''(Y)|]$, $\EE [ |f'(Y)|]$ and $\EE [|Y f(Y)|]$ are finite, we have
\begin{equation} \label{steinchgamma}
\EE \bigl[ \frac{1}{\lambda^2} Y f''(Y) + \frac{2 r}{\lambda^2} f'(Y) - Y f(Y) \bigr] =0.
\end{equation} 
\end{lemma}


Lemma \ref{gaunt} suggests the following Stein equation for the $\Gamma_s(\lambda,r)$-distribution:
\begin{equation} \label{steinidgamma}
\frac{1}{\lambda^2} x f''(x) + \frac{2 r}{\lambda^2} f'(x) - xf(x) = h(x) - \Gamma_s(\lambda, r)(h),
\end{equation}
where $\Gamma_s(\lambda, r)(h)$ denotes the quantity $\EE[h(Y)]$ with a random variable $Y$ distributed according to $\Gamma_s(\lambda,r)$.
The following lemma collects bounds on the solution $f_h$ of \eqref{steinidgamma} and its first and second derivative, see \cite[Theorem 3.6]{Gaunt:2013} for a proof. In what follows, we denote by $g^{(j)}$ the $j$th derivative of a function $g:\RR\to\RR$. 

\begin{lemma} \label{bounds}
Suppose that $h \in C_b^1(\RR)$, and $r \in \ZZ_+$ and $\lambda >0$, then the solution $f_h$ of the Stein equation \eqref{steinidgamma} and
its derivatives up to order two satisfy 
$$
\| f_h^{(j)} \|_{\infty} \leq c_j(\lambda,r)  \|h - \Gamma_s(\lambda, r)(h)\|_{\infty}, \quad j=0,1,
$$
$$
\| f_h^{(2)} \|_{\infty} \leq c_2^1(\lambda,r) \|h'\|_{\infty} + c_2^2(\lambda,r)  \|h - \Gamma_s(\lambda, r)(h)\|_{\infty}, 
$$
where
$c_0(\lambda,r) = \frac{1}{\sqrt{\lambda}} \Big( \frac 1r + \frac{\pi \Gamma(r/2)}{2 \Gamma(r/2+1/2)} \Big)$, $c_1(\lambda,r) = \frac{1}{\lambda} \Big( \frac 1r + \frac{1}{r+1} \Big)$, $c_2^1(\lambda,r)= \frac{3}{\lambda} \Big( \frac{\sqrt{\pi}}{\sqrt{2r+3}} + \frac 1r \Big)$ and $c_2^2(\lambda,r)= \frac{4}{\lambda^{3/2}} 
\Big( \frac{\sqrt{\pi}}{\sqrt{2r+3}} + \frac 1r \Big)$.
\end{lemma}

The Stein-type characterization \eqref{steinchgamma} for the $\Gamma_s(\lambda, r)$-distribution also allows a neat computation of its moments or cumulants. We state the result here only for the moments and cumulants of order $2$, $4$ and $6$ as they will play a major role later in this paper.

\begin{lemma} \label{gammamoments}
Let $Y$ be distributed according to $\Gamma_s(\lambda, r)$. Then all odd moments and cumulants of $Y$ are identically zero,
$$
\EE[Y^2]= \frac{2r}{\lambda^2}, \quad \EE[Y^4]=  \frac{12 r(r+1)}{\lambda^4}, \quad \EE[Y^6]=\frac{120 r(r+1)(r+2)}{\lambda^6}
$$
and
$$
\kappa_2(Y) = \frac{2r}{\lambda^2},\quad\kappa_4(Y) = \frac{12r}{\lambda^4},\quad\kappa_6(Y) = \frac{240r}{\lambda^6}.
$$
\end{lemma}

\begin{proof}
First, note that $\EE[Y^k]=0$ whenever $k\geq 1$ is an odd integer since $\Gamma_s(\lambda,r)$ is a symmetric distribution. Next, choosing $f(x)=x$ in \eqref{steinchgamma} we obtain $\EE[Y^2]= \frac{2r}{\lambda}$. Choosing $f(x) = x^3$ in  \eqref{steinchgamma} we get
$\frac{6(1+r)}{\lambda^2} \EE[Y^2] = \EE [Y^4]$ and with the choice $f(x)=x^5$ we obtain from  \eqref{steinchgamma} that
$\frac{10 r + 20}{\lambda^2} \EE [Y^4] = \EE [Y^6]$. The formulas for the cumulants follow from the relation between moments and cumulants stated in Section \ref{sec:Malliavin}.
\end{proof}

\paragraph{Variance-Gamma distributions.}
A random variable $Y$ is said to have a Variance-Gamma distribution with parameters $r>0$, $\theta \in \RR$, $\sigma >0$ and $\mu \in \RR$ if and only if its Lebesgue-density $p(x; r, \theta, \sigma, \mu)$, $x\in\RR$, equals
$$
p(x; r, \theta, \sigma, \mu) = \frac{1}{\sigma \sqrt{\pi} \Gamma(r/2)} \exp \biggl( \frac{\theta}{\sigma^2} (x-\mu) \biggr) \biggl( \frac{|x - \mu|}{2 \sqrt{\theta^2 + \sigma^2}} \biggr)^{\frac{r-1}{2}} K_{\frac{r-1}{2}} \biggl( \frac{\sqrt{\theta^2 + \sigma^2}}{\sigma^2} |x-\mu | \biggr).
$$
Here, $K_{\nu}(x)$ denotes a modified Bessel function of the second kind (see \cite[Appendix B]{Gaunt:2013} and references there). In what follows we write $VG(r,\theta, \sigma, \mu)$ for such a Variance-Gamma distribution. It is known that $\EE[Y] = \mu + r \theta$ and $\VV[Y] = r(\sigma^2 + 2 \theta^2)$. We will mostly consider only the centred case $\mu=0$ and write $VG_c(r,\theta, \sigma)$ for $VG(r,\theta,\sigma,0)$. Note that the symmetrized Gamma distribution considered in the previous paragraph corresponds to $VG(2r, 0, 1 / \lambda,0)$. Variance-Gamma distributions are widely used in finance modelling and contain as special or limiting cases the normal, Gamma  or Laplace distribution. In particular, for certain parameter values, the Variance-Gamma distribution has semi-heavy tails that decay slower than those of the normal distribution, see \cite{Gaunt:thesis,Gaunt:2013}. 

\begin{figure}[t]
\begin{center}
\includegraphics[width=0.45\columnwidth,clip, bb=70 320 525 785]{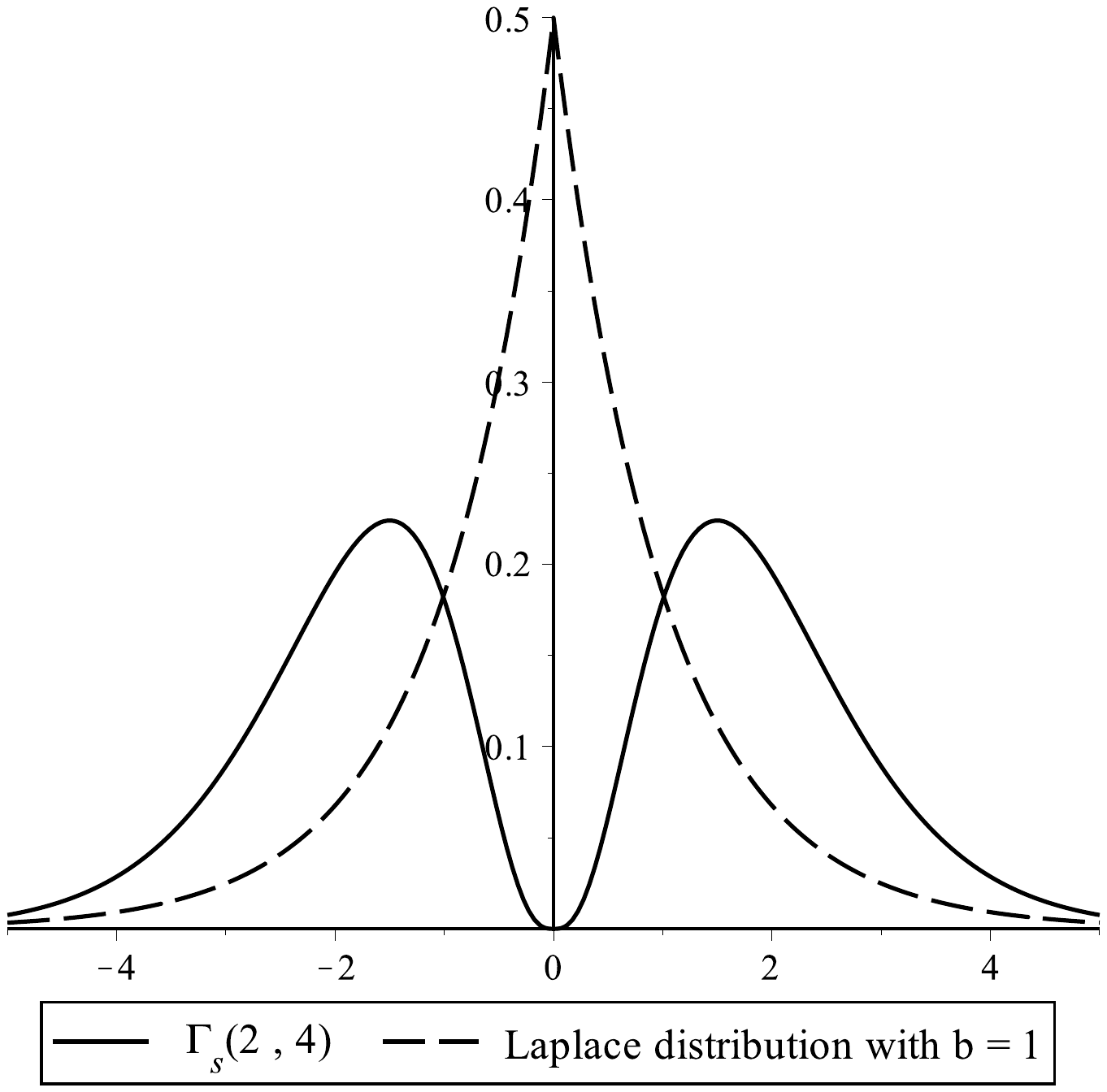}
\includegraphics[width=0.45\columnwidth,clip, bb=70 320 525 785]{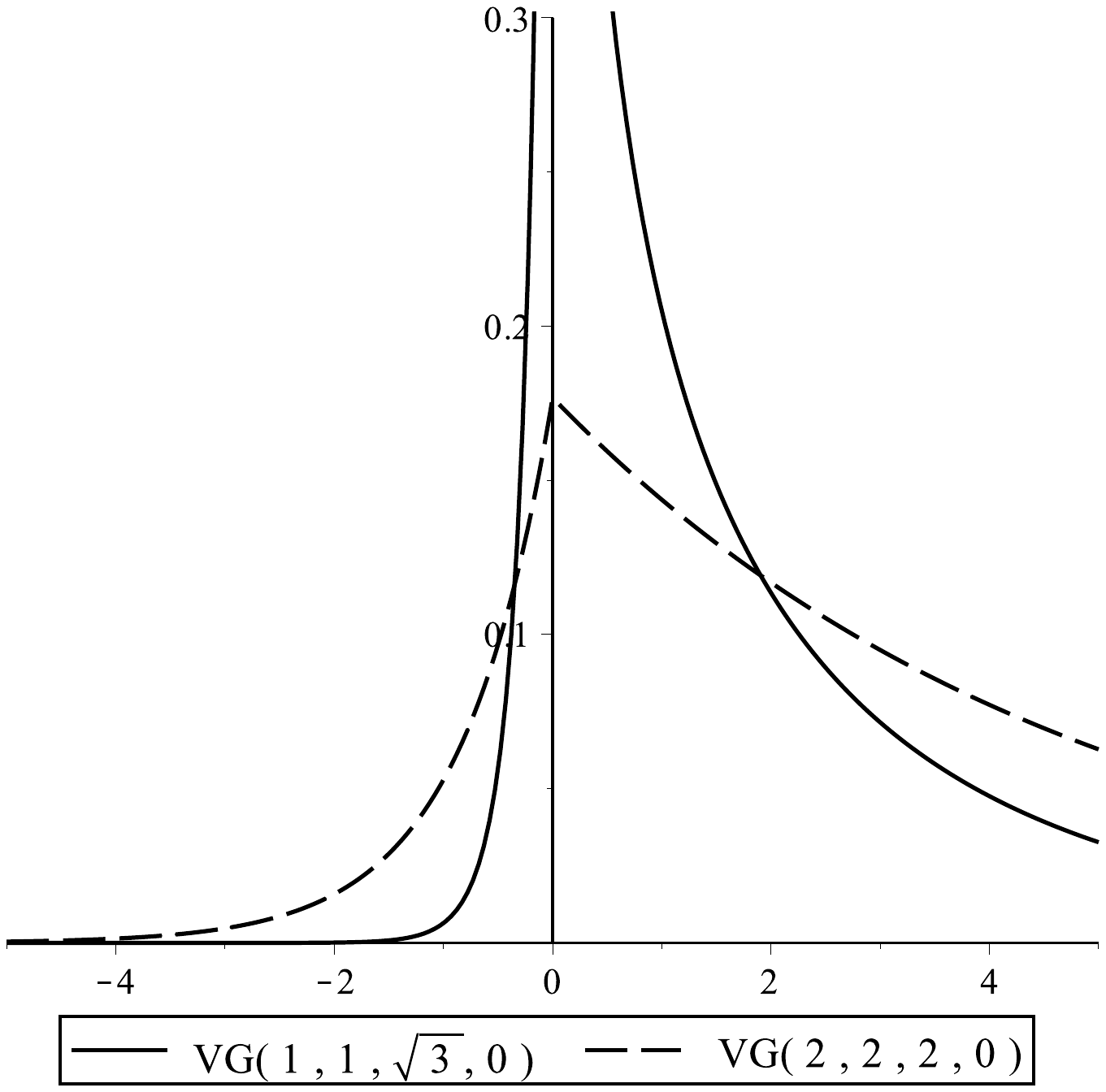}
\caption{Densities of Variance-Gamma distributions}
\label{fig:densities}
\end{center}
\end{figure}

The parameter $r$ is known to be the {\it scale} parameter. As $r$ increases, the distribution becomes more rounded around its peak value.
The parameter $\sigma$ is called the {\it tail} parameter. As $\sigma$ decreases, the tails drop off more steeply. Finally, the parameter $\theta$ is the {\it asymmetry}-parameter, for non-zero $\theta$ the distribution becomes skewed, that is, asymmetric, see Figure \ref{fig:densities}. In \cite{Gaunt:2013} a Stein equation for the $VG(r,\theta, \sigma, \mu)$-distribution was established.
From this, the Stein equation for the $VG_c(r,\theta, \sigma)$ distribution follows:
\begin{equation} \label{steinVGc}
\sigma^2 ( x + r \theta)  f''(x) + (\sigma^2 r + 2 \theta (x + r \theta)) f'(x) -x f(x) = h(x) - VG_c(r,\theta, \sigma)(h),
\end{equation}
where $VG_c(r,\theta, \sigma)(h)$ stands for the integral over $\RR$ of $h$ with respect to the $VG_c(r,\theta, \sigma)$ distribution.
The next lemma presents bounds for the solution $f_h$ of \eqref{steinVGc} and its first and second derivative. It is interesting to note that in contrast to the case $\theta=0$, uniform bounds are much harder to obtain if $\theta\neq 0$. In a first step these bounds can be expressed in terms of expressions involving modified Bessel functions, see Lemma 3.17 in \cite{Gaunt:thesis}. The following lemma follows from this representation.

\begin{lemma}\label{boundsVGc}
Suppose that $h \in C_b^1(\RR)$, and $r>0$, $\theta \in \RR$, $\sigma >0$, then the solution $f_h$ of the Stein equation \eqref{steinVGc} and
its derivatives up to order two are bounded, that is, there exists a constant $C=C(r,\theta,\sigma)$
such that 
\begin{equation}
\|f_h\|_{\infty}\leq C\|h\|_{\infty}, \quad
\|f_h^{(k)}\|_{\infty} \leq C\sum_{i=1}^{k-1}\|h^{(i)}\|_{\infty},\quad  k=1,2.
\end{equation}
\end{lemma}

\begin{remark}\rm
In contrast to the symmetric case discussed above, if $\theta\neq 0$, it seems rather difficult to express the constant $C$ appearing in Lemma \ref{boundsVGc} explicitly in terms of the parameters $r,\theta$ and $\sigma$. 
\end{remark}

With the same proof as for Lemma \ref{gammamoments} we can compute the first six moments or cumulants of a centred Variance-Gamma random variable, which will be needed later.

\begin{lemma} \label{VGcmoments}
If $Y$ is distributed according to $VG_c(r,\theta, \sigma)$, we obtain
$$
\EE[Y]=0,\quad\EE[Y^2]= r(\sigma^2 + 2 \theta^2), \quad \EE[Y^3]= 2  r \theta \sigma^2 + 4 \theta \, \EE[Y^2],
$$
$$
\EE[Y^4]=  \bigl( 3\sigma^2 (2+r) + 6 r \theta^2 \bigr) \, \EE[Y^2] + 6 \theta \, \EE[Y^3],
$$
$$
\EE[Y^5] = 12 r \theta  \sigma^2 \, \EE[Y^2] + \bigl( 8 r \theta^2 + 4r \sigma^2 + 12 \sigma^2 \bigr) \, \EE[Y^3] + 8 \theta \, \EE[Y^4],
$$
$$
\EE[Y^6]= 20 r \theta \sigma^2 \, \EE[Y^3] + \bigl(5\sigma^2(4+r) +10 r \theta^2 \bigr) \, \EE[Y^4] + 10 \, \theta \, \EE[Y^5].
$$
Moreover, the first six cumulants of $Y$ are $\kappa_1(Y)=0$ and
$$
\kappa_2(Y)=r(\sigma^2+2\theta^2),\quad\kappa_3(Y)=2r\theta(3\sigma^2+4\theta^2),\quad\kappa_4(Y)=6r(\sigma^4+8\sigma^2\theta^2+8\theta^4),
$$
$$
\kappa_5(Y)=24r\theta(5\sigma^4+20\sigma^2\theta^2+16\theta^4),\quad\kappa_6(Y)=120r(\sigma^2+2\theta^2)(\sigma^4+16\sigma^2\theta^2+16\theta^4).
$$
\end{lemma}

Let us collect some distributions, which are of particular interest and belong to the class of Variance-Gamma distributions, see \cite[Proposition 1.2]{Gaunt:2013}:
\begin{itemize}
\item A $VG_c(2r, 0, 1 / \lambda)$-distributed random variable has the symmetrized Gamma distribution, in particular $VG_c(2, 0, b)$ corresponds to a Laplace distribution with parameter $b$.
\item Suppose that $(X,Y)$ has the bivariate normal distribution with correlation $\varrho$ and marginals $X \sim \cN(0, \sigma_X^2)$ and 
$Y \sim \cN(0, \sigma_Y^2)$. Then the product $X \, Y$ follows the $VG_c(1, \varrho   \sigma_X  \sigma_Y,  \sigma_X  \sigma_Y \sqrt{1 - \varrho^2})$-distribution.
\item Suppose that $(X,Y)$ has the bivariate gamma distribution with correlation $\varrho$ and marginals $X \sim \Gamma(\lambda_1,r)$ and 
$Y \sim \Gamma(\lambda_2,r)$. Then the random variable $X - Y$ follows the $VG_c(2r, (2 \lambda_1)^{-1} -(2 \lambda_2)^{-1} ,  (\lambda_1 \lambda_2)^{-1/2} \sqrt{1 - \varrho^2})$-distribution.
\end{itemize}

\section{A Malliavin-Stein bound for the Wasserstein distance}\label{sec:GeneralBound}

Our first result provides explicit bounds for the $VG_c(r,\theta, \sigma)$-approximation of general functionals of an isonormal Gaussian process $X$.
Recall the definition of the $\Gamma$-operators $\Gamma_j(F)$ given in \eqref{gammadef}.

\begin{theorem} \label{result1}
Let $F \in \DD^{2,4}$ be such that $\EE[F]=0$ and let $Y$ be $VG_c(r,\theta, \sigma)$-distributed random variable. Then there exist constants $C_1=C_1(r,\theta, \sigma)>0$ and $C_2=C_2(r,\theta, \sigma) >0$ such that
\begin{equation} \label{r1}
d_W ( F, Y) \leq  C_1\EE \big[\big| \sigma^2(F + r\theta) + 2 \theta \, \Gamma_2(F) - \Gamma_3(F) \big|\big] + C_2 \big| r \sigma^2  + 2 r \theta^2 - \EE[\Gamma_2(F)] \big|.
\end{equation}
If in addition $F \in \DD^{3, 8}$, then $\Gamma_3(F)$ is square-integrable and 
\begin{equation} \label{r2}
\EE \big[ \big|  \sigma^2(F + r\theta) + 2 \theta \, \Gamma_2(F) - \Gamma_3(F)\big|\big] \leq \Big( \EE \big[\big( \sigma^2(F + r\theta) + 2 \theta \, \Gamma_2(F) - \Gamma_3(F) \big)^2 \big]\Big)^{1/2}.
\end{equation}
\end{theorem}

\begin{proof}
Let $f:\RR\to\RR$ be a twice differentiable function with bounded second derivative. 
Let $H=f(F)$ and put $G=F$. Then by our assumptions $H \in \DD^{1,2}$, using the chain rule \eqref{chainrule} and $G \in L^2(\Omega)$. Hence, by Lemma \ref{keyid} (1) we have that $$\EE [F f(F)] = \EE [f'(F) \Gamma_2(F)].$$ Similarly, let now $H=f'(F)$ and $G=F$, then by our assumptions $H \in \DD^{1,2}$, using the chain rule \eqref{chainrule}, and $G \in L^2(\Omega)$, which again 
by Lemma \ref{keyid} (1) leads to $$\EE [F f'(F)] = \EE [f''(F) \Gamma_2(F)].$$
Next we will apply Lemma \ref{keyid} (1) with $H=f'(F)$ and $G = \Gamma_2(F)$.
Again with \eqref{chainrule}, we have  that $f'(F) \in \DD^{1,2}$ and that $\Gamma_2(F)$ is square-integrable using $F \in \DD^{2,4} \subset \DD^{1,4}$ (for a detailed argument see \cite[Proof of Proposition 5.1.1]{NP:book}). Since $F \in \DD^{2,4}$, we have $\Gamma_3(F) \in L^1(\Omega)$ (see \cite[Lemma 4.2]{NourdinPeccati:2010}), whence
\begin{equation} \label{zwischenschritt}
 \EE [F f(F)] = \EE[f'(F)] \EE [\Gamma_2(F)] + \EE [f''(F) \Gamma_3(F)].
\end{equation}
Summarizing, we arrive at the identity
 \begin{eqnarray*}
 \EE \big[\sigma^2(F + r\theta) f''(F) + (\sigma^2 r + 2 r \theta^2)  f'(F) + 2 \theta F f'(F) - Ff(F) \big]  \\ 
 && \hspace{-9cm}
 =\EE \big[ f''(F) \big( \sigma^2(F + r\theta) + 2 \theta \, \Gamma_2(F) - \Gamma_3(F) \big) + f'(F) 
 \big( (r \sigma^2  + 2 r \theta^2) - \EE[\Gamma_2(F)] \big) \big]
 \end{eqnarray*}
and relation \eqref{r1} can be deduced from the bounds in Lemma \ref{bounds}. Relation \eqref{r2} is a consequence of \cite[Lemma 4.2(2)]{NourdinPeccati:2010}. Namely, with $F \in \DD^{3, 8}$ one has $\Gamma_3(F) \in \DD^{1,2}$.
 \end{proof}

\begin{remark}\rm
For a $VG_c(r,\theta, \sigma)$-distributed random variable $Y$ we know from Lemma \ref{VGcmoments} that  $\EE[Y^2] =r(\sigma^2 + 2 \theta^2) $. Since $\EE[\Gamma_2(F)] = \VV[F]=\EE[F^2]$, the second term in our bound \eqref{r1} measures the distance between the variances of $Y$ and $F$.  The interpretation of
 the $L^2$-distance of $\sigma^2(F + r\theta) + 2 \theta \, \Gamma_2(F)$ and the $\Gamma_3(F)$-term on the right-hand side of \eqref{r2} is not obvious and will be discussed for $F\in\cH_q$ being in the $q$th Wiener chaos in Section \ref{sec:multipleintegrals} below.
 \end{remark}
 
 We will now derive two consequences from Theorem \ref{result1}. The first one deals with two special Variance-Gamma distributions, the symmetric Gamma distribution $\Gamma_s(\lambda,r)$ and the distribution of $X-Y$ of two random variables $X$ and $Y$ having a $\Gamma_s(\lambda_1,r)$- and $\Gamma_s(\lambda_2,r)$-distribution, respectively.
 
\begin{corollary} \label{result1b}
Let $F \in \DD^{2,4}$ be such that $\EE[F]=0$. 
\begin{itemize}
\item[(a)] Let $Y$ be a $VG_c(2r, 0, 1 / \lambda)=\Gamma_s(\lambda,r)$-distributed random variable for some $\lambda,r>0$. Then 
\begin{equation*} 
d_W ( F, Y ) \leq C_1\Big( \EE\Big[ \Big( \frac{1}{\lambda^2} F - \Gamma_3(F) \Big)^2\Big] \Big)^{1/2} + C_2\Big| \frac{2r}{\lambda^2} - \EE[\Gamma_2(F)] \Big|
\end{equation*}
with constants $C_1,C_2>0$ only depending on $\lambda$ and $r$.
\item[(b)] Fix $r,\lambda_1,\lambda_2,\varrho>0$ and let $Z$ denote a real-valued random variable with a $VC_c(2r, (2 \lambda_1)^{-1} -(2 \lambda_2)^{-1} ,  (\lambda_1 \lambda_2)^{-1/2} \sqrt{1 - \varrho^2})$-distribution. Then
\begin{eqnarray*} 
d_W ( F, Z ) & \leq & C_1 \Big( \EE \Big[\Big( \frac{1 - \varrho^2}{\lambda_1 \lambda_2} \Big( F + r \Big( \frac{1}{\lambda_1} - \frac{1}{\lambda_2}\Big) \Big) + \Big( \frac{1}{\lambda_1} - \frac{1}{\lambda_2}\Big) \Gamma_2(F) - \Gamma_3(F) \bigg)^2\, \Big] \Big)^{1/2}\\
&  &\qquad + C_2 \big| \EE[Z] - \EE[\Gamma_2(F)] \big|
\end{eqnarray*}
with constants $C_1,C_2>0$ depending only on $r,\lambda_1,\lambda_2$ and $\varrho$.
\end{itemize}
\end{corollary}

Our next result deals with two limiting cases of Variance-Gamma distributions, namely the normal and the (non-symmetrized) Gamma distribution. As discussed in the introduction, this has previously been considered in \cite{NourdinPeccati:2009}. More precisely, Theorems 3.1 and 3.11 there show that if $F\in\DD^{1,2}$ is a centred functional of an isonormal Gaussian process and if $Z\sim \cN(0,\sigma^2)$ for some $\sigma^2>0$ and $Y\sim\Gamma(\lambda,r)$ for some $\lambda,r>0$ that
\begin{align*}
d_W(F,Z) \leq \big(\EE[(\sigma^2-\Gamma_2(F))^2]\big)^{1/2}\quad\text{and}\quad d_W(F,Y)\leq C\, \Big(\EE\Big[\Big({1\over\lambda}F+{r\over\lambda^2}-\Gamma_2(F)\Big)^2\,\Big]\Big)^2
\end{align*}
with a constant $C>0$ only depending on $r$ and $\lambda$. In our context, we can derive another bound for $d_W(F,Z)$ and $d_W(F,Y)$ in terms of the Gamma-operator $\Gamma_3$. We will see below that in the case of multiple stochastic integrals this is closely related to some of the results recently derived in \cite{AzmoodehEtAl:2013}.

\begin{corollary} \label{result1c}
Let $F \in \DD^{2,4}$ be such that $\EE[F]=0$.
\begin{itemize}
\item[(a)] Let $Z$ denote a centred Gaussian random variable with variance $\sigma^2>0$. Then there exist  constants $C_1,C_2>0$ only depending on $\sigma$ such that
\begin{equation} \label{r1c}
d_W ( F, Z) \leq  C_1 \EE\big[ \big| \Gamma_3(F) \big|\big] + C_2  \big| \sigma^2   - \EE[\Gamma_2(F)] \big|.
\end{equation}
\item[(b)] Let $Y$ be a $\Gamma(\lambda,r)$-distributed random variable with parameters $\lambda >0$ and $r>0$. Then there exist constants $ C_1,C_2>0$ depending only on $r$ and $\lambda$ such that
\begin{equation} \label{r2c}
d_W ( F, \Gamma(\lambda,r) ) \leq  C_1\EE \Big[\Big| \frac{1}{\lambda} \Gamma_2(F) - \Gamma_3(F) \Big|\Big] + C_2  \Big| \frac{r}{\lambda^2}   - \EE[\Gamma_2(F)] \Big|.
\end{equation}
\end{itemize}
\end{corollary}
\begin{proof}
We apply Theorem \ref{result1} and use the fact that
$$\lim_{r \to \infty} VG_c(r,0, \sigma/ \sqrt{r}) = \cN(0, \sigma^2)\quad\text{and}\quad\lim_{\sigma \to 0} VG_c(2r, \frac{1}{2 \lambda}, \sigma) = \Gamma(\lambda,r),$$
see for example \cite[Proposition 2.6 (i) and (iv)]{Gaunt:2013}. Hence, with \eqref{r1} we have to consider $(\sigma^2/r) F - \Gamma_3(F)$, which is converging to $\Gamma_3(F)$, as $r \to \infty$. In case of a Gamma distribution we have to consider $\frac{1}{\lambda} \Gamma_2(F) - \Gamma_3(F)$.
\end{proof}

\begin{remark}\rm
In case (a) of Corollary \ref{result1c} one is able to get the same bound (with different constants) for the Kolmogorov-distance, see \cite[Theorem 5.1.3]{NP:book}. It is interesting to compare our bound in \eqref{r1c} with (5.1.3) and (5.1.5) in \cite{NP:book}. While we have to consider $\EE[ | \Gamma_3(F) |]$, the estimate in \cite{NP:book} reads $(\VV[\Gamma_2(F)])^{1/2}$. As explained earlier, this comes from the fact that we consider the much larger class of Variance-Gamma distributions based on a second order-differential equation. This also implies that the stronger condition $F \in \DD^{2,4}$ is needed.
\end{remark}

\section{Explicit bounds on a fixed Wiener chaos}\label{sec:multipleintegrals}
 
\subsection{The general case $q\geq 2$}\label{subsec:Generalqgeq2}

Fix $q \geq 2$ and consider $F_n = I_q(f_n)$, $n \geq 1$, a sequence of random variables belonging to the $q$th chaos of an isonormal Gaussian process $X$ and assume that $\EE[F_n^2] = q! \|f_n\|_{\EH^{\otimes q}}^2 \to r(\sigma^2 + 2 \theta^2)$ with $r >0$, $\sigma>0$ and $\theta \in \RR$. The sequence $\{F_n\}_{n\in\NN}$ converges in distribution
to $Y \sim VG_c(r,\theta,\sigma)$, if and only if for every $j \geq 3$, $\EE[F_n^j] \to \EE[Y^j]$, as $n\to\infty$, or equivalently if
$\kappa_j(F_n) \to \kappa_j(Y)$ for every $j \geq 3$, as $n\to\infty$. This follows from the classical method of moments or cumulants, since the law $VG_c(r,\theta,\sigma)$ is determined by its moments (compare with Proposition 5.2.2 in \cite{NP:book}). 

One of our main result is, that the method of moments and cumulants for $VG_c(r,\theta,\sigma)$-approxima\-tion boils down to a sixth-moment method inside the second Wiener chaos, see Section \ref{double}. For general $q \geq 2$ the next result provides an expression for the first term of the bound in Theorem \ref{result1} in terms of contraction operators. Note that if $q \geq 3$ is an odd integer and $\theta \not= 0$, then there is no sequence $\{F_n\}_{n\in\NN} = \{I_q(f_n)\}_{n\in\NN}$, such that $F_n$ has bounded variances and such that $F_n$ converges in distribution to a random variable $Y$ with a $VG_c(r,\theta,\sigma)$-distribution, as $n \to \infty$. This is a consequence of the fact that an element of a Wiener chaos of odd order has its third moment equal zero, while $\EE[Y^3] =  \theta ( 2  r \sigma^2 + 4 \, \EE[Y^2]) \not= 0$ whenever $\theta \not=0$.

\begin{theorem} \label{thm:ContractionConditions}
Let $q \geq 2$ be an even integer and let $F=I_q(f)$, where $f \in \EH^{\odot q}$. Then we have
\begin{eqnarray*}
&& \EE\big[ \bigl( \sigma^2(I_q(f) + r\theta) + 2 \theta \, \Gamma_2(I_q(f)) - \Gamma_3(I_q(f)) \big)^2\big]  = \biggl( \frac 12 \EE [I_q(f)^3] - ( 2 \theta \EE [I_q(f)^2] + r \theta \sigma^2) \biggr)^2 \\
&& \qquad +\ q! \bigg\| \sum_{r=1}^{q-1} c_q(r, q-r)  ((f \widetilde{\otimes}_r f) \widetilde{\otimes}_{q-r} f) - 2 \theta c_q(q/2) ( f \widetilde{\otimes}_{q/2} f ) - \sigma^2 f \bigg\|_{\EH^{\otimes q}}^2 \\
&& \qquad + \sum_{k=1, k \not= q/2}^{q-1} (2k)! \big\| \bigl( g_k(f,q) - 2 \theta c_q(q-k) ( f \widetilde{\otimes}_{q-k} f ) \bigr) \big\|_{\EH^{\otimes 2k}}^2 + 
 \sum_{k=q}^{\frac{3q}{2} -2} (2k)! \big\| g_k(f,q) \big\|_{\EH^{\otimes 2k}}^2.
\end{eqnarray*}
In case $q=2$ the last two sums are empty and have to be interpreted as $0$.
\end{theorem}
\begin{proof}
We start with the observation that \eqref{gammarep} for $s=2$ leads to
\begin{equation*} 
\Gamma_3(I_q(f)) = \sum_{r=1}^{q-1} \sum_{s=1}^{(2q-2r) \wedge q} c_q(r,s) \, I_{3q -2r-2s} ((f \widetilde{\otimes}_r f) \widetilde{\otimes}_s f).
\end{equation*}
Next, we rewrite $\Gamma_3(I_q(f))$. For this, let $q$ be even and put $2k = 3q -2r-2s$ and $C_{2k} := \{ r \in \{1, \ldots, q-1 \}: 0 \leq \frac{3q}{2} -k-r \leq (2q-2r) \wedge q \}$, the set of those integers $r$ for which the so-called double contraction $(f \widetilde{\otimes}_r f) \widetilde{\otimes}_{\frac{3q}{2} -k-r} f$ is well-defined. Then,
\begin{equation} \label{gamma3}
\Gamma_3(I_q(f)) = \sum_{k=0}^{\frac{3q}{2} -2} I_{2k} \biggl( \sum_{r \in C_{2k}} c_q(r, 3q/2 -k-r) 
((f \widetilde{\otimes}_r f) \widetilde{\otimes}_{\frac{3q}{2} -k-r} f) \biggr) =:\sum_{k=0}^{\frac{3q}{2} -2} I_{2k} \bigl( g_k(f,q) \bigr) .
\end{equation}

For  $F=I_q(f)$, with $f \in \EH^{\odot q}$, we know that $\Gamma_2(F)= \langle DF, -DL^{-1} F \rangle_{\EH} = q^{-1} \| DF \|_{\EH}^2$. Combining this with \cite[Equation (5.2.2)]{NP:book} and the notation introduced around \eqref{gammarep} we find that
\begin{eqnarray} \label{gamma2}
\Gamma_2(I_q(f)) &=& q \sum_{r=1}^q (r-1)! {q-1 \choose r-1}^2 I_{2q-2r} ( f \widetilde{\otimes}_r f )  
\nonumber \\
& = &q! \|f\|_{\EH^{\otimes q}}^2 + 
\sum_{r=1}^{q-1} c_q(r) I_{2q-2r} ( f \widetilde{\otimes}_r f ) . 
\end{eqnarray}
Due to the multiplication formulae \eqref{productrule} one obtains that, for even $q$, we have
\begin{equation} \label{thirdmoment}
\EE [I_q(f)^3] = q! (q/2)! {q \choose q/2}^2 \langle f, f \widetilde{\otimes}_{q/2} f \rangle_{\EH^{\otimes q}}.
\end{equation}
 According to \eqref{r2} in Theorem \ref{result1}, we have to compute 
\begin{equation} \label{startvonallem}
\Gamma_3(I_q(f)) - 2 \theta \, \Gamma_2(I_q(f)) - \sigma^2(I_q(f) + r\theta)
\end{equation} for $\theta \in \RR$, $r>0$ and $\sigma>0$. At first, we collect the constant terms. In \eqref{gamma3}, we obtain for
$k=0$  that $C_{2 0} = \{ q/2 \}$ and therefore the constant term is $I_{2\cdot 0} (g_0(f,q)) = g_0(f,q) = c_q(q/2, q) ((f \widetilde{\otimes}_{q/2} f) \widetilde{\otimes}_{q} f)$. With
the definition of $c_q(q/2, q)$ in \eqref{gammarep} and \eqref{thirdmoment} we obtain $I_{2\cdot 0} (g_0(f,q)) = \frac 12 \EE [I_q(f)^3]$. Hence, the constant in \eqref{startvonallem} equals
\begin{equation} \label{constants}
\frac 12 \EE [I_q(f)^3] - \bigl( 2 \theta \EE [I_q(f)^2] + r \theta \sigma^2 \bigr),
\end{equation}
using $q! \|f\|_{\EH^{\otimes q}}^2=\EE [I_q(f)^2]$.
Next, we consider the so called middle-contractions in $\Gamma_2(I_q(f))$ and $\Gamma_3(I_q(f))$, i.e., contractions of order $q/2$. With $r = q/2$ in \eqref{gamma2} we obtain the term $c_q(q/2) I_{q} ( f \widetilde{\otimes}_{q/2} f )$ and with $k=q/2$ in \eqref{gamma3} we get $C_{2 \frac q2} = \{ 1, \ldots, q-1\}$ and hence the term $\sum_{r=1}^{q-1} c_q(r, q-r)  ((f \widetilde{\otimes}_r f) \widetilde{\otimes}_{q-r} f)$. Summarizing, the middle-contraction in \eqref{startvonallem} contributes
\begin{equation} \label{middlecont}
I_q \biggl( \sum_{r=1}^{q-1} c_q(r, q-r)  ((f \widetilde{\otimes}_r f) \widetilde{\otimes}_{q-r} f) - 2 \theta c_q(q/2) ( f \widetilde{\otimes}_{q/2} f ) - \sigma^2 f \biggr).
\end{equation}
The remaining terms in \eqref{startvonallem} can be represented as follows:
\begin{eqnarray} \label{rest}
\sum_{k=1, k \not= q/2}^{\frac{3q}{2} -2} I_{2k} \bigl( g_k(f,q) \bigr) - 2 \theta \sum_{r=1, r\not= q/2}^{q-1} c_q(r) I_{2q-2r} ( f \widetilde{\otimes}_r f ) \nonumber  \\
& & \hspace{-8cm} =
\sum_{k=1, k \not= q/2}^{q-1} I_{2k} \biggl( g_k(f,q) - 2 \theta c_q(q-k) ( f \widetilde{\otimes}_{q-k} f ) \biggr) +  \sum_{k=q}^{\frac{3q}{2} -2} I_{2k} \bigl( g_k(f,q) \bigr)
\end{eqnarray}
With \eqref{constants}, \eqref{middlecont} and \eqref{rest} we obtain that \eqref{startvonallem} is equal to
\begin{eqnarray*}
&& I_q \biggl( \sum_{r=1}^{q-1} c_q(r, q-r)  ((f \widetilde{\otimes}_r f) \widetilde{\otimes}_{q-r} f) - 2 \theta c_q(q/2) ( f \widetilde{\otimes}_{q/2} f ) - \sigma^2 f \biggr)
\\ && \qquad+ \sum_{k=1, k \not= q/2}^{q-1} I_{2k} \bigl( g_k(f,q) - 2 \theta c_q(q-k) ( f \widetilde{\otimes}_{q-k} f ) \bigr) +  \sum_{k=q}^{\frac{3q}{2} -2} I_{2k} \bigl( g_k(f,q) \bigr) \\ 
&& \qquad+ \frac 12 \EE [I_q(f)^3] - ( 2 \theta \EE [I_q(f)^2] + r \theta \sigma^2).
\end{eqnarray*}
Using the isometric property \eqref{isometry} of multiple Wiener integrals we can now conclude the result.
\end{proof}

Let us have a closer look at the first summand $\frac 12 \EE [I_q(f)^3] - ( 2 \theta \EE [I_q(f)^2] + r \theta \sigma^2)$ appearing in the expression provided by Theorem \ref{thm:ContractionConditions}. Using Lemma \ref{VGcmoments} we see that the moment assumption that $\EE [I_q(f_n)^2]$ and $\EE [I_q(f_n)^3]$ converge to $\EE[Y^2]$ and $\EE[Y^3]$, respectively, ensures that $\frac 12 \EE [I_q(f)^3] - ( 2 \theta \EE [I_q(f)^2] + r \theta \sigma^2)$ converges to zero, as $n\to\infty$. Note moreover that the other contraction operators do not depend on $r$. The dependence on $r$ is completely encoded in the moment assumption that $\EE [I_q(f_n)^2] \to \EE[Y^2]$ and $\EE [I_q(f_n)^3] \to \EE[Y^3]$.

\medskip

Next we consider the particularly attractive case $\theta=0$ separately corresponding to the symmerized Gamma distributions separately. As explained earlier, in this case no restriction on the parity of $q$ is necessary.

\begin{theorem} \label{result2}
Let $q \geq 2$ be an integer and let $F=I_q(f)$ with $f \in \EH^{\odot q}$. Then for $q$ being even we have
\begin{eqnarray*}
\EE \Big[\Big( \frac{1}{\lambda^2} F - \Gamma_3(F) \Big)^2\Big] & = & q! \Big\| \frac{1}{\lambda^2} f - \sum_{r=1}^{q-1} 
c_q(r, q-r) ((f \widetilde{\otimes}_r f) \widetilde{\otimes}_{q-r} f) \Big\|_{\EH^{\otimes q}}^2 \\
&+& \sum_{k=0, k \not= q/2}^{\frac{3q}{2}-2} 
(2k)! \Big\| \sum_{r \in C_{2k}} c_q(r, 3q/2-k-r) ((f \widetilde{\otimes}_r f) \widetilde{\otimes}_{\frac{3q}{2} -k-r} f) \Big\|_{\EH^{\otimes 2k}}^2,
\end{eqnarray*} 
whereas for $q$ being odd we set $p=q-1$ and obtain
\begin{eqnarray*}
\EE \Big[\Big( \frac{1}{\lambda^2} F - \Gamma_3(F) \Big)^2\Big] & = & q! \Big\| \frac{1}{\lambda^2} f - \sum_{r=1}^{q-1} 
c_q(r, q-r) ((f \widetilde{\otimes}_r f) \widetilde{\otimes}_{q-r} f) \Big\|_{\EH^{\otimes q}}^2 \\
&+& \sum_{k=0, k \not= p/2}^{\frac{3p}{2}-1} 
(2k)! \Big\| \sum_{r \in C_{2k+1}} c_q(r, 3p/2+1-k-r) ((f \widetilde{\otimes}_r f) \widetilde{\otimes}_{\frac{3p}{2}+1 -k-r} f) \Big\|_{\EH^{\otimes 2k+1}}^2.
\end{eqnarray*} 
\end{theorem}

\begin{proof}
For $q$ being even the result follows directly form Theorem \ref{thm:ContractionConditions}. The case when $q \geq 3$ is odd is similar. 
Here, we put $p :=q-1$ and denote by $C_{2k+1}$ the set of those integers $r \in \{1, \ldots, q-1\}$ for which the double contraction $((f \widetilde{\otimes}_r f) \widetilde{\otimes}_{3q/2+1-k-r} f)$ is well defined. We skip the details.
\end{proof}

\begin{remark}\rm
The symmetric Gamma distribution can be presented as a finite linear combination of independent chi-squared random variables. In \cite[Theorem  3.2]{APP:2014}
the authors investigated necessary and sufficient conditions for convergence in distribution towards such a combination within the framework of random objects living on a fixed chaos. In the examples in \cite[Section 4]{APP:2014}, the conditions are presented in terms of contractions.
\end{remark}

\begin{remark}\label{rem:tetilla}\rm 
Comparing the contraction conditions implied by Theorems \ref{result1} and \ref{result2} for the symmetric Gamma distribution with those of Theorem 1.1 (ii) in \cite{DeyaNourdin:2012} for the tetilla law arising in free probability we see that our condition in the case of $\Gamma_s(1, \frac{1}{\sqrt{2}})$ 
coincides almost readily with that in \cite{DeyaNourdin:2012}. The only difference are the coefficients $c_q(r,q-r)$, which arise as a consequence of the product formula \eqref{productrule}. In contrast, these coefficients are all equal to $1$ in the free set-up (compare with Equation (2.6) in \cite{DeyaNourdin:2012}, for example). This way, we may identify the Laplace distribution with parameter $\sqrt{2}$ as the non-free analogue of the tetilla law.
\end{remark}

A particularly interesting question is whether the bounds derived in Theorems \ref{thm:ContractionConditions} and \ref{result2} are tight with respect to the convergence in distribution towards a Variance-Gamma distribution, in the sense that these bounds converge to zero if and only if a normalised sequence $\{ F_n\}_{n \in\NN}$, living inside a fixed Wiener chaos, converges in distribution to a $VG_c(r, \theta, \sigma)$-distributed random variable. Fix $q \geq 2$, and consider a sequence $\{ F_n: n \geq 1\}$ such that $F_n = I_q(f_n)$, $n \geq 1$, where $f_n \in \EH^{\odot q}$ and suppose that $\EE[F_n^2] = q! \| f_n \|_{\EH^{\otimes q}}^2 \to \frac{2r}{\lambda^2}$. Moreover, by $Y$ denote a random variable with $\Gamma_s(\lambda,r)$-distribution. We conjecture that for the symmetric Variance-Gamma distributions (corresponding to $\theta=0$) (i) the convergence in distribution of $F_n$ to $Y$ is equivalent to (ii) $\EE[F_n^4] \to \EE[Y^4]$ and $\EE[F_n^6] \to \EE[Y^6]$, which in turn is equivalent to the contraction conditions (iii) that 
$$\| ((f_n \widetilde{\otimes}_r f_n) \widetilde{\otimes}_{r'} f_n) \|_{\EH^{\otimes 3q-2r-2r'}} \to 0\quad\text{and}\quad \Big\| \frac{1}{\lambda^2} f_n - \sum_{r=1}^{q-1} 
c_q(r, q-r) ((f_n \widetilde{\otimes}_r f_n) \widetilde{\otimes}_{q-r} f_n) \Big\|_{\EH^{\otimes q}} \to 0,
$$
where $r = 1, \ldots, q-1$ and $r'$ is such that $r' + 2r \leq 2q$ and $r+r' \not=q$. Our conjecture for $\theta \not= 0$ reads similar. Namely, we conjecture that a sequence $\{F_n\}_{n\in\NN}$ such that $F_n = I_q(f_n)$, $n \geq 1$, where $f_n \in \EH^{\odot q}$ and $\EE[F_n^2] = q! \| f_n \|_{\EH^{\otimes q}}^2 \to r(\sigma^2 + 2 \theta^2)$ (i) converges in distribution to a $VG_c(r, \theta, \sigma)$-distributed random variable if and only if (ii) the moment condition $\EE[F_n^j] \to \EE[Y^j]$ is satisfied for $j=3,4,5,6$ or if and only if (iii) the contraction conditions $\| ((f_n \widetilde{\otimes}_l f_n) \widetilde{\otimes}_{3q/2-k-l} f_n) \|_{\EH^{\otimes 3q-2r-2r'}} \to 0$  for every $l = 1, \ldots, 3q/2-k-1$ and $k=q, \ldots, 3q/2 -2$,
$$
\bigg\| \sum_{r=1}^{q-1} c_q(r, q-r)  ((f \widetilde{\otimes}_r f) \widetilde{\otimes}_{q-r} f) - 2 \theta c_q(q/2) ( f \widetilde{\otimes}_{q/2} f ) - \sigma^2 f \bigg\|_{\EH^{\otimes q}} \to 0
$$
and $\big\| \bigl( g_k(f,q) - 2 \theta c_q(q-k) ( f \widetilde{\otimes}_{q-k} f ) \bigr) \big\|_{\EH^{\otimes 2k}} \to 0$ for every $k \in \{1, \ldots, q-1\} \setminus \{q/2 \}$ hold.

The technically sophisticated step in both situations is to show that (ii) implies (iii). The main difficulty is to deal with the involved combinatorial structure transmitted from the product formula to the collection of double contractions. In Section \ref{double} below, we will obtain a positive answer to both of the above stated conjectures in the particular case $q=2$, while general case remains open, because for general $q$ we were not able to express (or to estimate from above) the bounds of Theorems \ref{thm:ContractionConditions} or \ref{result2} in terms of the first six moments of the involved chaotic random variables.

\medskip
The following discussion concerns the symmetric Gamma approximation of a finite sum of Wiener chaoses. Without loss of generality we discuss a sum of two Wiener chaoses. Consider two integers $2 \leq q_1 < q_2$ and a sequence of the form
$$
Z_n = I_{q_1}(f_n^1) +  I_{q_2}(f_n^2), \quad n \geq 1,
$$
where $f_n^i \in \EH^{\odot q_i}$. In order to bound the second summand on the right hand side of \eqref{r1} we have to compute $\EE[\Gamma_2(Z_n)]$.
By the product formula \eqref{productrule} we obtain
$$
\EE[\Gamma_2(Z_n)] = q_1 ! \| f_n^1 \|_{\EH^{\otimes q_1}}^2 + q_2 ! \| f_n^2 \|_{\EH^{\otimes q_2}}^2.
$$
Hence to ensure convergence of $\EE[\Gamma_2(Z_n)]$ it is not necessary to that each of the summands $\EE[\Gamma_2(I_{q_i})] = q_i ! \| f_n^i \|_{\EH^{\otimes q_i}}^2$ converges. Next, we have to bound $\EE \bigl[ (\frac{1}{\lambda^2} Z_n - \Gamma_3(Z_n))^2 \bigr]$. Without loss of generality, we can assume that $X$ is an isonormal process over a Hilbert space of the type $L^2(A, {\mathcal A}, \mu)$. For every $b \in A$, it is immediately checked that 
$$
-D_b L^{-1} \Gamma_2(Z_n) = \sum_{i,j=1}^2 q_i \sum_{r=1}^{q_i \wedge q_j} (r-1)! {q_i -1 \choose r-1} {q_j -1 \choose r-1} I_{q_i + q_j - 2r -1} (f_n^i \widetilde{\otimes}_r f_n^j( \cdot, b))
$$
and $D_b Z_n = q_1 I_{q_1-1} (f_n^1(\cdot,b)) + q_2 I_{q_2-1} (f_n^2(\cdot,b))$. Therefore, by the product formula,
\begin{eqnarray*}
\Gamma_3(Z_n)   \!  \! \!  \! \! \! \!  & = &\! \! \! \!  \! \! \! \! \! \! \sum_{i,j,k=1}^2 \sum_{r=1}^{q_i \wedge q_j} q_i q_j  (r-1)! {q_i -1 \choose r-1} {q_j -1 \choose r-1} \int_A I_{q_i-1} (f_n^i(\cdot,b)) \, I_{q_k +q_j - 2r -1} \bigl( (f_n^j \widetilde{\otimes}_r f_n^k)(\cdot,b) \bigr) \mu(db) \\
& = & \! \! \! \!  \! \! \! \! \! \! \sum_{i,j,k=1}^2 \sum_{r=1}^{q_i \wedge q_j} \sum_{s=1}^{q_i \wedge (q_j+q_k -2r)} q_i q_j  (r-1)! {q_i -1 \choose r-1} {q_j -1 \choose r-1}  (s-1)! 
{q_i -1 \choose s-1} {q_j +q_k -2r -1 \choose s-1} \\
& \qquad \times & I_{q_i+q_j+q_k -2r-2s}   \bigl( f_n^i \widetilde{\otimes}_s (f_n^j \widetilde{\otimes}_r f_n^k) \bigr) =:  \sum_{i,j,k=1}^2 \sum_{r=1}^{q_i \wedge q_j} \sum_{s=1}^{q_i \wedge (q_j+q_k -2r)} T(q_i,q_j,q_k,r,s, f_n^1, f_n^2).
\end{eqnarray*}
Now, we consider the two summands $i=j=k=1$ and $i=j=k=2$ and choose $s=q_i-r$. We observe that these summands can be re-presented as 
$\sum_{r=1}^{q_l-1} c_{q_l} (r, q_l-r) \, I_{q_l} \bigl( f_n^l \widetilde{\otimes}_{q-r} (f_n^l \widetilde{\otimes}_r f_n^l) \bigr)$ for $l=1,2$. Summarizing, we have
\begin{eqnarray*}
\frac{1}{\lambda^2} Z_n - \Gamma_3(Z_n) & = &\sum_{l=1,2} I_{q_l} \biggl( \frac{1}{\lambda^2} f_n^l  - \sum_{r=1}^{q_l-1} c_{q_l}(r, q_l-r) 
\bigl( f_n^l \widetilde{\otimes}_{q-r} (f_n^l \widetilde{\otimes}_r f_n^l) \bigr) \biggr) \\
&\qquad +& \sum_{(i,j,k,r,s) \in {\mathcal S}} T(q_i,q_j,q_k,r,s, f_n^1, f_n^2)
\end{eqnarray*}
with
\begin{eqnarray*}
{\mathcal S} & := & \bigl\{ (i,j,k,r,s) \in \{1,2\}^3 \times \NN^2 : 1 \leq r \leq q_i \wedge q_j, 1 \leq s \leq q_i \wedge q_j+q_k -2r \,\, \text{and, whenever} \\
& & \,\, i=j=k, r \not= q_i \,\, \text{and} \,\, s \not= q_i-r \bigr\}. 
\end{eqnarray*}
By using the inequality $(a_1 + a_2)^2 \leq 2 (a_1^2 + a_2^2)$ and the isometric property \eqref{isometry}  we obtain:

\begin{proposition} \label{sum}
Consider two integers $2 \leq q_1 < q_2$ and a sequence of the form
$$
Z_n = I_{q_1}(f_n^1) +  I_{q_2}(f_n^2), \quad n \geq 1,
$$
where $f_n^i \in \EH^{\odot q_i}$. Then for every $\lambda >0$ we have
\begin{eqnarray*}
& & \EE \bigl[ (\frac{1}{\lambda^2} Z_n - \Gamma_3(Z_n))^2 \bigr] \\
& \leq & 8 \sum_{l=1,2} \bigg\| \frac{1}{\lambda^2} f_n^l  - \sum_{r=1}^{q_l-1} c_{q_l}(r, q_l-r)  
\bigl( f_n^l \widetilde{\otimes}_{q-r} (f_n^l \widetilde{\otimes}_r f_n^l) \bigr) \bigg\|_{\EH^{\otimes q_l}}^2 \\
& + & 
\sum_{(i,j,k,r,s) \in {\mathcal S}} q_i q_j  (r-1)! {q_i -1 \choose r-1} {q_j -1 \choose r-1}  (s-1)! 
{q_i -1 \choose s-1} {q_j +q_k -2r-1 \choose s-1} \bigl\| f_n^i \widetilde{\otimes}_s (f_n^j \widetilde{\otimes}_r f_n^k) \bigr\|^2.  
\end{eqnarray*}
\end{proposition}

Using Proposition \ref{sum}, it is in principle also possible to deduce bounds for the Variance-Gamma approximation
of random variables living inside an infinite sum of Wiener chaoses.
\medskip

We finally turn in this section to the case of normal approximation and recover the celebrated fourth moment theorem. Moreover, our more general framework implies the following result, which leads to a better rate of convergence (namely exponent 3/2 instead of 1) compared with \cite[Theorem 5.2.7]{NP:book}, for example. However, our rate is still not optimal as shown by the main result in \cite{NourdinPeccati:2013}.

\begin{proposition}\label{prop:GaussContractions}
Fix $q \geq 2$, and consider a sequence $\{ F_n\}_{n\in\NN}$ such that $F_n = I_q(f_n)$, $n \geq 1$, where $f_n \in \EH^{\odot q}$. Assume that $\EE[F_n^2] = \sigma^2>0$ and $\EE[\Gamma_3(F_n)^2] \to 0$, as $n \to \infty$. Then the sequence $\{F_n\}_{n\in\NN}$ satisfies a central limit theorem and we have the following bound for the Wasserstein distance: $$d_W(F_n,Z)\leq C\,\max_{r=1,\ldots,q-1\atop r'+2r\leq 2q}\{\| ((f_n \widetilde{\otimes}_r f_n) \widetilde{\otimes}_{r'} f_n) \|_{\EH^{\otimes 3q-2r-2r'}} \} \leq C\,\max_{1 \leq l \leq q-1}\{\| f_n \otimes_l f_n  \|_{\EH^{\otimes 2q-2l}}^{3/2}\},$$ where $C>0$ is a constant only depending on $\sigma$ and where $Z\sim \cN(0,\sigma^2)$. Moreover, we have that $\EE[\Gamma_3(F_n)^2] \to 0$ if and only if $\kappa_4(F_n)\to 0$, as $n\to\infty$
\end{proposition}
\begin{proof}
That the sequence $\{F_n\}_{n\in\NN}$ satisfies a central limit theorem under our assumptions is ensured by Corollary \ref{result1c} (a). Moreover, using the multiplication formula \eqref{productrule} we have
\begin{equation} \label{gamma3quad}
\EE[\Gamma_3^2(I_q(f))] = \sum_{k=0}^{\frac{3q}{2}-2} 
(2k)! \bigl\| g_k(f,q) \big\|_{\EH^{\otimes 2k}}^2.
\end{equation}
Hence a sufficient condition for a central limit theorem to hold is that for every 
$r = 1, \ldots, q-1$ and $r'$ such that $r' + 2r \leq 2q$ it holds that 
$$\| ((f_n \widetilde{\otimes}_r f_n) \widetilde{\otimes}_{r'} f_n) \|_{\EH^{\otimes 3q-2r-2r'}} \to 0,
$$
as $n\to \infty$.  Now, the double-contractions are dominated by the usual (single) contractions in the following way:
$$
\| ((f_n \widetilde{\otimes}_r f_n) \widetilde{\otimes}_{r'} f_n) \|_{\EH^{\otimes 3q-2r-2r'}}  \leq \max_{1 \leq l \leq q-1} 
\| f_n \otimes_l f_n  \|_{\EH^{\otimes 2q-2l}}^{3/2},
$$
see \cite[Equation (4.10)]{BBNP:2012}. This proves the first part the result.

As shown above, the sequence $\{F_n\}_{n\in\NN}$ satisfies a central limit theorem provided that $\EE[\Gamma_3(F_n)^2] \to 0$. By the fourth moment theorem \cite[Theorem 5.2.7]{NP:book}, the central limit theorem for $\{F_n\}_{n\in\NN}$ is equivalent to the condition that $\kappa_4(F_n)\to 0$. This proves the second part of the result.
\end{proof}

 \subsection{The case of the second Wiener chaos}\label{double}
 
The goal of this subsection is to confirm the two conjectures spelled out in the previous subsection for elements of the second Wiener chaos (i.e., for double stochastic integrals). That is, we consider a sequence of elements of the second Wiener chaos of an isonormal process $X$, that is, a sequence of random variables of the type $F_n = I_2(f_n)$ with $f_n \in \EH^{\odot 2}$ for each $n\in\NN$. For symmetric Variance-Gamma distributions ($\theta=0$) our result reads as follows.

\begin{theorem}\label{thm:DoubleIntegralsRingTheta=0}
Let $Y$ be a $\Gamma_s(\lambda,r)$-distributed random variable with $r,\lambda>0$ and suppose that $\EE[F_n^2]=2r/\lambda^2$. Then, as $n\to\infty$, the following assertions are equivalent:
\begin{itemize}
\item[(a)] $F_n=I_2(f_n)$ converges in distribution to $Y$,
\item[(b)] $\EE[F_n^4]\to\EE[Y^4]$ and $\EE[F_n^6]\to\EE[Y^6]$,
\item[(c)] $\| 4 ((f_n \widetilde{\otimes}_1 f_n)  \widetilde{\otimes}_1 f_n)  - \frac{1}{\lambda^2}f_n \|_{\EH^{\otimes 2}} \to 0$ and $\| ((f_n \widetilde{\otimes}_1 f_n)  \widetilde{\otimes}_2 f_n) \|^2 \to 0$.
\end{itemize}
\end{theorem}

In the general asymmetric case $\theta\neq 0$, stronger moment or contraction conditions are necessary in order to ensure convergence in distribution of $F_n$ to a Variance-Gamma distributed random variable.

\begin{theorem}\label{thm:DoubleIntegralsRingThetaNot0}
Let $Y$ be a $VG_c(r,\theta,\sigma)$-distributed random variable with $r,\sigma>0$ and $\theta\in\RR$, and suppose that $\EE[F_n^2]=r(\sigma^2+2\theta^2)$. Then, as $n\to\infty$, following assertions are equivalent:
\begin{itemize}
\item[(a)] $F_n=I_2(f_n)$ converges in distribution to $Y$,
\item[(b)] $\EE[F_n^j]\to\EE[Y^j]$ for all $j=3,4,5,6$,
\item[(c)] $\| 4 ((f_n \widetilde{\otimes}_1 f_n)  \widetilde{\otimes}_1 f_n)  - 2 \theta \, (f_n \widetilde{\otimes}_1 f_n)\ - \sigma^2 f_n \|_{\EH^{\otimes 2}} \to 0$ 
and $\| ((f_n \widetilde{\otimes}_1 f_n)  \widetilde{\otimes}_2 f_n) \|_{\EH^{\otimes 2}} \to \frac 34  r \theta \sigma^2 + r \theta^3$.
\end{itemize}
\end{theorem}

Before entering the proofs of Theorems \ref{thm:DoubleIntegralsRingTheta=0} and \ref{thm:DoubleIntegralsRingThetaNot0}, we collect some general facts about random variables of the type $F=I_2(f)$, $f\in\EH^{\odot 2}$, belonging to the second Wiener chaos $\cH_2$ and introduce some notation. First recall that the law of $F$ is determined by its moments or, equivalently, by its cumulants. The latter are given by
\begin{equation} \label{cumq2}
 \kappa_p(F) =2^{p-1} (p-1)! \langle f \otimes_1^{(p-1)} f, f \rangle_{\EH^{\otimes 2}},\qquad p\geq 2,
 \end{equation}
thanks to relation \eqref{eq:KappaUndGamma}. Here, $\{ f \otimes_1^{(p)} f: p \geq 1 \} \subset \EH^{\odot 2}$ is the sequence defined by $f  \otimes_1^{(1)} f =f$ and for $p \geq 2$ by $f \otimes_1^{(p)} f = \bigl( f \otimes_1^{(p-1)} f \bigr) \otimes_1 f$. In particular $f \otimes_1^{(2)} f = f \otimes_1 f$.

 \begin{proof}[Proof of Theorem \ref{thm:DoubleIntegralsRingThetaNot0}]
The implication (a) $\Rightarrow$ (b) is trivial and (c) $\Rightarrow$ (a) follows by combining Theorem \ref{result1} with Theorem \ref{thm:ContractionConditions}. Thus, it remains to show that (b) implies (c).

Let  $F_n = I_2(f_n)$ with $f_n \in \EH^{\odot 2}$, $n \geq 1$.
Theorem \ref{result1} for $q=2$ leads to
\begin{equation} \label{compare0}
\begin{split}
\EE[ (\Gamma_3(F_n) - 2 \theta \Gamma_2(F_n) - \sigma^2 (F_n + r \theta))^2] & = 2 \|   4 ((f_n \widetilde{\otimes}_1 f_n) \widetilde{\otimes}_1 f_n)  - 4 \theta 
(f_n \widetilde{\otimes}_1 f_n) - \sigma^2 \,  f \|_{\EH^{\otimes 2}}^2  \\ 
& \qquad+ \Big( \frac 12 \EE [F_n^3] - ( 2 \theta \EE [F_n^2] + r \theta \sigma^2) \Big)^2.
\end{split}
\end{equation}
For $\theta =0$, $\sigma = 1/ \lambda$, with \eqref{thirdmoment} we obtain
\begin{equation} \label{compare1}
\EE[ (\Gamma_3(F_n) - \frac{1}{\lambda^2} F_n)^2] = 2 \|   \frac{1}{\lambda^2} \, f - 4 (f_n \widetilde{\otimes}_1 f_n) \widetilde{\otimes}_{1} f_n \|_{\EH^{\otimes 2}}^2 + 16  \langle f_n \otimes_1 f_n, f_n \rangle_{\EH^{\otimes 2}}^2.
\end{equation}
We represent the left hand side of \eqref{compare0} in terms of moments and cumulants of $F_n$ to be able to check that if the six moment
condition on $F_n$ (condition (b) in Theorem \ref{thm:DoubleIntegralsRingThetaNot0}) is satisfied, then condition (c) for the contractions follows.  The left-hand side of \eqref{compare0} consists of six terms. Identity \eqref{gamma3quad} gives
$$\EE[\Gamma_3^2(F_n)] = 2^5 \|(f_n \widetilde{\otimes}_1 f_n) \widetilde{\otimes}_{1} f_n) \|_{\EH^{\otimes 2}}^2  + 16 \langle f_n \otimes_1 f_n, f_n \rangle_{\EH^{\otimes 2}}^2.$$ By \eqref{cumq2} we obtain $\kappa_6(F_n) = 2^5 5! 
\langle f \otimes_1^{(5)} f, f \rangle_{\EH^{\otimes 2}} = 2^5 5! \|(f_n \widetilde{\otimes}_1 f_n) \widetilde{\otimes}_{1} f_n) \|_{\EH^{\otimes 2}}^2$, 
implying that
\begin{equation*} \label{gamma32double}
\EE[\Gamma_3^2(F_n)] = \frac{1}{120} \kappa_6(F_n) + 16 
\langle f_n \otimes_1 f_n, f_n \rangle_{\EH^{\otimes 2}}^2 =  \frac{1}{120} \kappa_6(F_n) + \frac 14 (\kappa_3(F_n))^2.
\end{equation*}
Next, with \eqref{gamma2} we get 
\begin{eqnarray*}
4\theta^2 \EE[\Gamma_2^2(F_n)] & = & 32 \theta^2 \| f_n \widetilde{\otimes}_1 f_n \|^2 +  4 \theta^2 \kappa_2(F_n)^2
\\ & = & 32 \theta^2 \langle  f_n \otimes_1^{(3)} f_n, f_n \rangle +  4 \theta^2 \kappa_2(F_n)^2= \frac 23 \theta^2 \kappa_4(F_n)+  4 \theta^2 \kappa_2(F_n)^2,
\end{eqnarray*}
 using that $\kappa_4(F_n) = 48 \langle 
f_n \otimes_1^{(3)} f_n, f_n \rangle$, see \eqref{cumq2}. For the third term we have $\EE[\sigma^4 (F_n + r \theta)^2] = \sigma^4 \EE[F_n^2] + r^2 \theta^2 \sigma^4$.
Applying Part (1) of Lemma \ref{keyid} with $s=1$ we obtain for the fourth term $$4 \theta \sigma^2 \EE[F_n \Gamma_2(F_n)] + 4 r \theta^2 \sigma^2 \EE[\Gamma_2(F_n)]=
2 \theta \sigma^2 \EE[F_n^3] + 4 r \theta^2 \sigma^2 \EE[F_n^2].$$ With $\EE[\Gamma_3(F_n)] = \frac 12 \kappa_3(F_n)$, the fifth term 
reads $-2 \sigma^2 \EE[F_n \Gamma_3(F_n)] - r \theta \sigma^2 \EE[F_n^3]$.
Part (1) of Lemma \ref{keyid} implies $\EE[I_q(f)^2 \Gamma_2(I_q(f))] = \EE[I_q(f)^2]
\EE[\Gamma_2(I_q(f))] + 2 \EE[I_q(f) \Gamma_3(I_q(f))]$ and part (2) says that $$\EE[I_q(f)^2 \Gamma_2(I_q(f))]= q \EE [I_q(f)^2 \|D(I_q(f)\|_{\EH}^2]= \frac 13 \EE[I_q(f)^4].$$ Hence $\frac 13 \EE[I_q(f)^4] = \EE[I_q(f)^2]^2 + 2 \EE[I_q(f) \Gamma_3(I_q(f))]$,
and it follows that $\EE[I_q(f) \Gamma_3(I_q(f))] = \frac 16 \kappa_4(I_q(f))$.
Hence the fifth term can be presented as $$- \frac 13 \sigma^2 \kappa_4(F_n) - r \theta \sigma^2 \EE[F_n^3].$$
Finally, we have to compute $-4 \theta \EE[\Gamma_2(F_n) \, \Gamma_3(F_n)]$.
With \eqref{gamma3}, \eqref{gamma2} and \eqref{cumq2} we obtain
\begin{eqnarray*}
-4 \theta \EE[\Gamma_2(F_n) \, \Gamma_3(F_n)] & = & -64 \theta \langle f_n \widetilde{\otimes}_1 f_n, (f_n \widetilde{\otimes}_1 f_n) \widetilde{\otimes}_1 f_n \rangle - 2 \theta \kappa_2(F_n) \, \kappa_3(F_n) \\
&= &-64 \theta \langle 
f_n \otimes_1^{(4)} f_n, f_n \rangle - 2 \theta \kappa_2(F_n) \, \kappa_3(F_n) = - \frac{\theta}{6} \, \kappa_5(F_n) -2 \theta \kappa_2(F_n) \, \kappa_3(F_n).
\end{eqnarray*}
Summarizing, the left hand side of \eqref{compare0} is equal to
\begin{equation} \label{VGcboundcum}
\begin{split}
\frac{1}{120} \kappa_6(F_n) &-\frac{\theta}{6}  \kappa_5(F_n) + \frac 13 (2 \theta^2 - \sigma^2)  \kappa_4(F_n)  + (2-r) \theta \sigma^2  \kappa_3(F_n) + \frac 14 ( \kappa_3(F_n))^2  \\
& -  2 \theta \kappa_2(F_n) \kappa_3(F_n)  + (\sigma^4 + 4r \theta^2 \sigma^2)  \kappa_2(F_n) + 4 \theta^2 (\kappa_2(F_n))^2 + r^2 \theta^2 \sigma^4.
\end{split}
\end{equation}
Using now the moments assumption (b) together with Lemma \ref{VGcmoments}, we see that the term in \eqref{VGcboundcum} converges to zero as $n \to \infty$ and hence the contraction condition (c) follows, see \eqref{compare0}. This completes the proof.
 \end{proof}

\begin{proof}[Proof of Theorem \ref{thm:DoubleIntegralsRingTheta=0}]
As in the asymmetric case, it suffices to show that (b) implies (c). In our case, $\theta=0$ and we put $\sigma = \frac{1}{\lambda}$ and obtain that
\begin{equation} \label{compare2}
\EE[ (\Gamma_3(F_n) - \frac{1}{\lambda^2} F_n)^2] = \frac{1}{120} \kappa_6(F_n) - \frac{1}{3 \lambda^2}  \kappa_4(F_n) + \frac 14 ( \kappa_3(F_n))^2
+ \frac{1}{\lambda^4} \kappa_2(F_n)
\end{equation}
from \eqref{compare0} and \eqref{VGcboundcum}.
Hence with \eqref{compare1} and  \eqref{compare2}  we get
$$
2 \|   \frac{1}{\lambda^2} \, f - 4 (f_n \widetilde{\otimes}_1 f_n) \widetilde{\otimes}_{1} f_n \|_{\EH^{\otimes 2}}^2 = \frac{1}{120} \kappa_6(F_n) - 
 \frac{1}{3 \lambda^2}  \kappa_4(F_n) + \frac{1}{\lambda^4} \kappa_2(F_n).$$
Into the last identity we plug the well known relationships between moments and cumulants stated in Section \ref{sec:Malliavin}.
Then, a simple calculation leads to 
\begin{eqnarray*}
2 \| \frac{1}{\lambda^2} f - 4 (f_n \widetilde{\otimes}_1 f_n) \widetilde{\otimes}_{1} f_n) \|_{\EH^{\otimes 2}}^2 \\
& & \hspace{-5cm} = \frac{1}{120} \EE[F_n^6] - \biggl( \frac 18 + \frac{1}{6r} \biggr) \EE[F_n^2] \EE[F_n^4] + \biggl( \frac 14 + \frac{1}{2r} + \frac{1}{(2r)^2} \biggr) \EE[F_n^2]^3 - \frac{1}{12} \EE[F_n^3]^2.
\end{eqnarray*}
Now, we assume that $\EE[F_n^2] \to \frac{2r}{\lambda^2}$, $\EE[F_n^4] \to \frac{12r(r+1)}{\lambda^4}$ and $\EE[F_n^6] \to \frac{120r(r+1)(r+2)}{\lambda^6}$.
Then,
$$ 
\frac{1}{120} \EE[F_n^6] - \biggl( \frac 18 + \frac{1}{6r} \biggr) \EE[F_n^2] \EE[F_n^4] + \biggl( \frac 14 + \frac{1}{2r} + \frac{1}{(2r)^2} \biggr) \EE[F_n^2]^3  \to 0,
$$
as $n \to \infty$. Since $\EE[F_n^3]^2 = 64 \| (f \widetilde{\otimes}_1 f) \widetilde{\otimes}_2 f \|_{\EH^{\otimes 2}}^2$, the contraction conditions in (c) follow.
\end{proof}

\begin{remark} \rm \label{contrationcollection}
Let  $F_n = I_2(f_n)$ with $f_n \in \EH^{\odot 2}$, $n \geq 1$. Assume that $\EE[F_n^2] = q! \| f_n \|_{\EH^{\otimes q}}^2 \to r(\sigma^2 + 2 \theta^2)$.
Here we list the different forms of conditions on contraction-operators which are equivalent to the convergence in distribution to a member of $VG_c(r, \theta, \sigma)$.
\begin{enumerate}
\item
$F_n$ converges to $\cN(0, \sigma^2)$ if and only if $\| f_n \otimes_1 f_n \|_{\EH^{\otimes 2}} \to 0$,
as $ n \to \infty$. 
\item
$F_n$ converges to $\Gamma(\lambda,r)$ if and only if $\| f_n \widetilde{\otimes}_1 f_n  - \frac{1}{2 \lambda}f_n \|_{\EH^{\otimes 2}} \to 0$, as $ n \to \infty$.
\item
$F_n$ converges to $\Gamma_s(\lambda,r)$ if and only if 
$\| 4 ((f_n \widetilde{\otimes}_1 f_n)  \widetilde{\otimes}_1 f_n)  - \frac{1}{\lambda^2}f_n \|_{\EH^{\otimes 2}} \to 0$ and $\| ((f_n \widetilde{\otimes}_1 f_n)  \widetilde{\otimes}_2 f_n) \|^2 \to 0$, as $ n \to \infty$.
\item
$F_n$ converges to $VG_c(r, \theta, \sigma)$ if and only if 
$\| 4 ((f_n \widetilde{\otimes}_1 f_n)  \widetilde{\otimes}_1 f_n)  - 2 \theta \, (f_n \widetilde{\otimes}_1 f_n)\ - \sigma^2 f_n \|_{\EH^{\otimes 2}} \to 0$ 
and $\| ((f_n \widetilde{\otimes}_1 f_n)  \widetilde{\otimes}_2 f_n) \|_{\EH^{\otimes 2}} \to \frac 34  r \theta \sigma^2 + r \theta^3$, as $ n \to \infty$.
\item
An example of case (d) is the convergence to $VG_c(1,\varrho,\sqrt{1 - \varrho^2})$, which can be interpreted as the distribution of the product of two correlated standard normal
distributed random variables $X$ and $Y$ with correlation $\varrho$. We obtain that
$F_n$ converges to $VG_c(1,\varrho,\sqrt{1 - \varrho^2})$ if and only if 
$\| 4 ((f_n \widetilde{\otimes}_1 f_n)  \widetilde{\otimes}_1 f_n)  - 2 \varrho \, (f_n \widetilde{\otimes}_1 f_n)\ -  (1 - \varrho^2) f_n \|_{\EH^{\otimes 2}} \to 0$ 
and $\| ((f_n \widetilde{\otimes}_1 f_n)  \widetilde{\otimes}_2 f_n) \|_{\EH^{\otimes 2}} \to \frac 34 \varrho (1-\varrho^2) + \varrho^3$, as $ n \to \infty$.
When $\varrho \to 0$, case (c) appears with $\lambda=r=1$.
\end{enumerate}
\end{remark}

After having characterized convergence in distribution of an element belonging to the second Wiener chaos $\cH_2$, we turn now to quantitative bounds for the Wasserstein distance. In contrast to the bounds that follow from the results presented in Section \ref{sec:GeneralBound} and Section \ref{subsec:Generalqgeq2}, we are seeking for upper bounds in terms of moments. In view of Theorems \ref{thm:DoubleIntegralsRingTheta=0} and \ref{thm:DoubleIntegralsRingThetaNot0} we can expect that these bounds only involve moments up to order six. Our next theorem presents bounds in terms of the first six cumulants, as they have a more compact form.
 
 \begin{theorem} \label{doublebounds}
 Let  $F_n = I_2(f_n)$ with $f_n \in \EH^{\odot 2}$, $n \geq 1$. 
 \begin{itemize}
 \item[(a)]
Let $Y$ denote a $VG_c(r,\theta,\sigma)$-distributed random variable and assume that $\EE[F_n^2] = 2 \| f_n \|_{\EH^{\otimes 2}}^2 \to r(\sigma^2 + 2 \theta^2)$. Then there exist constants $C_1=C_1(r, \theta,\sigma)>0$ and $C_2=C_2(r, \theta,\sigma) >0$ such that
\begin{eqnarray*}
d_W ( F_n,Y)\!\!\!\! &\leq&  \!\!\!\!C_1
\Big(  \frac{1}{120} \kappa_6(F_n) - \frac{\theta}{6}  \kappa_5(F_n) + \frac 13 (2 \theta^2 - \sigma^2)  \kappa_4(F_n)  + (2-r) \theta \sigma^2  \kappa_3(F_n) + \frac 14 ( \kappa_3(F_n))^2 \\
&& \qquad - 2 \theta \kappa_2(F_n) \kappa_3(F_n)  + (\sigma^4 + 4r \theta^2 \sigma^2)  \kappa_2(F_n) + 4 \theta^2 (\kappa_2(F_n))^2 + r^2 \theta^2 \sigma^4 \Big)^{1/2} \\
&& \qquad+ C_2 \big| r(\sigma^2 + 2 \theta^2) - \kappa_2(F_n) \big|.
\end{eqnarray*} 
\item[(b)]
 Let $Y$ be $\Gamma_s(\lambda,r)$-distributed random variable and assume that $\EE[F_n^2] = 2 \| f_n \|_{\EH^{\otimes 2}}^2 \to \frac{2r}{\lambda^2}$. 
 Then there are constants $C_1=C_1(\lambda,r)>0$ and $C_2= C_2(\lambda,r) >0$ such that
\begin{eqnarray*}
d_W ( F_n, Y) &\leq& C_1
\Big( \frac{1}{120} \kappa_6(F_n) - \frac{1}{6r} \kappa_4(F_n) \kappa_2(F_n) + \frac{1}{4r^2} \kappa_2(F_n)^3 +  \frac 16 \kappa_3(F_n)^2 \Big)^{1/2}\\
&& \qquad +C_2 \Big| \frac{2r}{\lambda^2} - \kappa_2(F_n) \Big|.
\end{eqnarray*} 
\end{itemize}
 \end{theorem}
 
 \begin{remark} \rm
 The bound in Theorem \ref{doublebounds} (b) suggests that we have -- in addition to the convergence of the second, fourth and sixth moment or cumulant -- to assume that also the third moment or cumulant of $F_n$
 converge to zero, as $n \to \infty$, to conclude convergence in distribution to the limiting random variable.  However, we
 know already from Theorem \ref{thm:DoubleIntegralsRingTheta=0} that convergence of the second, fourth and sixth moments or cumulants suffices to obtain convergence in law  and hence we can conclude that the third moment of $F_n$ converges to zero automatically under these conditions.
 \end{remark}
 
 We turn now to the case of normal approximation, which appears as a limiting case of a Variance-Gamma distribution, see Proposition \ref{prop:GaussContractions}. Our next result provides a bound for the Wasserstein distance between second chaos element and a Gaussian random variable in terms of the second, third and sixth cumulant. It implies that sequence $F_n=I_2(f_n)$ converges in distribution to a $\cN(0,\sigma^2)$-distributed random variable ($\sigma^2>0$) if $\EE[F_n^2] \to \sigma^2$ and if $\kappa_3(F_n)\to 0$ and $\kappa_6(F_n)\to 0$, as $n\to\infty$. Clearly, this is weaker than the usual forth moment theorem for which we refer to \cite{NP:book}.

 \begin{proposition} \label{gauss2}
 Let  $F_n = I_2(f_n)$ with $f_n \in \EH^{\odot 2}$, $n \geq 1$.  There exists constants $C_1(\sigma), C_2(\sigma) >0$ such that
 $$
d_W ( F, \cN(0, \sigma^2)) \leq  C_1(\sigma) \biggl( \frac{1}{120} \kappa_6(F_n) + \frac 14 \kappa_3(F_n)^2\biggr) + C_2(\sigma)  \big| \sigma^2   - \EE[F_n^2] \big|.
$$
The same bound holds for the Kolmogorov-distance with different constants.
 \end{proposition}
 
 \begin{proof}
This is a direct consequence of Corollary \ref{result1c} and the identity \eqref{gamma32double} for $\EE[\Gamma_3(F_n)^2]$.
 \end{proof}

The next statement provides a further characterization of the convergence of the elements of the second chaos, i.e. for $F=I_2(f)$ with $f\in\EH^{\odot 2}$. To avoid technical complications, we restrict for the rest of this section to a symmetrized Gamma distribution $\Gamma_s(\lambda,r)$. To state the result, consider the Hilbert-Schmidt operator $A_f : \EH \to \EH$, $g \mapsto f \otimes_1 g$ associated with $f$ and write $\{ \lambda_{f,j} : j \geq 1 \}$ and  $\{ e_{f,j} : j \geq 1 \}$, respectively, for the eigenvalues of $A_f$ and the corresponding eigenvectors. It is well known (see \cite[Section 2.7.4]{NP:book}) that, the series
 $\sum_{j \geq 1} \lambda_{f,j}^p$ converges for all $p \geq 2$, and that $f$ admits the expansion (in $\EH^{\odot 2}$)
 \begin{equation} \label{expansionHS}
 f = \sum_{j \geq 1} \lambda_{f,j} \, \,  \big( e_{f,j} \otimes e_{f,j} \big).
 \end{equation}
We notice that for the trace of the $p$th power of $A_f$ one has the relation ${\rm Tr}(A_f^p) = \langle f \otimes_1^{(p-1)} f, f \rangle_{\EH^{\otimes 2}} = \sum_{j \geq 1}   \lambda_{f,j}^p$.

\begin{theorem} \label{furthercharc}
 Let  $F_n = I_2(f_n)$ with $f_n \in \EH^{\odot 2}$, $n \geq 1$.
Let $Y$ denote a random variable with $\Gamma_s(\lambda,r)$-distribution assume that
$\EE[F_n^2] = 2 \| f_n \|_{\EH^{\otimes q}}^2 \to \frac{2r}{\lambda^2}$. Then the following two conditions are equivalent to the conditions stated in Theorem \ref{thm:DoubleIntegralsRingTheta=0}:
\begin{enumerate}
\item As $n\to\infty$, $\sum_{j \geq 1} \bigl( \frac{1}{\lambda^2} \lambda_{f_n,j} - 4 \lambda_{f_n,j}^3 \bigl)^2 \to 0$ and $\sum_{j \geq 1} \lambda_{f_n,j}^3 \to 0$, where, for each $n \geq 1$, $\{ \lambda_{f_n,j} \,\}_{j \geq 1}$ stands for the sequence of the eigenvalues of the Hilbert-Schmidt operator $A_{f_n}$. 
\item As $n\to\infty$, $\sum_{j \geq 1} \lambda_{f_n,j}^3 \to 0$ and for every $q \geq 2$,
\begin{equation} \label{ind}
\sum_{j \geq 1} \lambda_{f_n,j}^{2 q} \to \frac{2r}{\lambda^2} \, \Big( \frac{1}{4 \lambda^2} \Big)^{q-1}.
\end{equation}
\end{enumerate} 
\end{theorem}

\begin{proof}
To prove the equivalence of (a) to (c) in Theorem \ref{thm:DoubleIntegralsRingTheta=0}, we use \eqref{expansionHS} to deduce that
$$
(f_n \otimes_1 f_n) \otimes_1 f_n = \sum_{j \geq 1} \lambda_{f_n,j}^3 \, \,  \bigl( e_{f,j} \otimes e_{f,j} \bigr) \,\, \text{and} \,\, 
(f_n \otimes_1 f_n) \widetilde{\otimes}_2 f_n = \sum_{j \geq 1} \lambda_{f_n,j}^3.
$$
It follows that
$$
\|   \frac{1}{\lambda^2} \, f_n - 4 (f_n \widetilde{\otimes}_1 f_n) \widetilde{\otimes}_{1} f_n \|_{\EH^{\otimes 2}}^2 = \sum_{j \geq 1} \bigl( \frac{1}{\lambda^2} \lambda_{f_n,j} - 4 \lambda_{f_n,j}^3 \bigl)^2.
$$
Next we show that (a) is equivalent to (b). The proof of the implication  $(a) \Longrightarrow (b)$ is based on a recursive argument. By assumption we have
$\sum_{j \geq 1} \lambda_{f_n,j}^2 \to \frac{2r}{\lambda^2}$. Moreover,
\begin{eqnarray*}
\bigg| \sum_{j\geq 1} \lambda_{f_n,j}  \bigl( \frac{1}{\lambda^2} \lambda_{f_n,j} - 4 \lambda_{f_n,j}^3 \bigl) \bigg| 
\leq  \biggl( \sum_{j\geq 1} \lambda_{f_n,j}^2 \biggr)^{1/2}  \biggl( \sum_{j\geq 1} \bigl( \frac{1}{\lambda^2} \lambda_{f_n,j} - 4 \lambda_{f_n,j}^3 \bigl)^2 \biggr)^{1/2} \to 0,
\end{eqnarray*}
thus yielding that  $\lim_{n \to \infty} \sum_{j\geq 1} \lambda_{f_n,j}^4 = \frac{1}{4 \lambda^2} \lim\limits_{n \to \infty} \sum_{j\geq 1} \lambda_{f_n,j}^2 = \frac{1}{4 \lambda^2} \, \frac{2r}{\lambda^2}$. Now, if \eqref{ind} holds, then
\begin{eqnarray*}
\bigg| \sum_{j\geq 1} \lambda_{f_n,j}^{2q-1}  \bigl( \frac{1}{\lambda^2} \lambda_{f_n,j} - 4 \lambda_{f_n,j}^3 \bigl) \bigg| 
\leq  \biggl( \sum_{j\geq 1} \lambda_{f_n,j}^{4q-2} \biggr)^{1/2}  \biggl( \sum_{j\geq 1} \bigl( \frac{1}{\lambda^2} \lambda_{f_n,j} - 4 \lambda_{f_n,j}^3 \bigl)^2 \biggr)^{1/2} \to 0,
\end{eqnarray*}
and \eqref{ind} with $q$ replaced by $q+1$ follows. To see the implication $(b) \Longrightarrow (a)$, just write
$$
\sum_{j \geq 1} \bigl( \frac{1}{\lambda^2} \lambda_{f_n,j} - 4 \lambda_{f_n,j}^3 \bigl)^2 = \frac{1}{\lambda^4} \sum_{j \geq 1}  \lambda_{f_n,j}^2 - \frac{8}{\lambda^2}  \sum_{j \geq 1}  \lambda_{f_n,j}^4 + 16 \sum_{j \geq 1}  \lambda_{f_n,j}^6
$$
and apply \eqref{ind} with $q=1,2,3$.
\end{proof}

As a consequence of Theorem \ref{furthercharc} we deduce the following characterization of a symmetrized Gamma random variable in the second Wiener chaos.

\begin{corollary} \label{char}
Fix an integer $n \in \NN$. Let $I_2(f)$ with $f \in \EH^{\odot 2}$ be such that $\EE[I_2(f)^2] = 4n$. Then the following conditions are equivalent:
\begin{enumerate}
\item $I_2(f)$ is distributed according to $\Gamma_s( \frac 12, \frac n2)$.
\item $\EE[I_2(f)^4]= \EE[ Y(n)^4]$ and $\EE[I_2(f)^6]= \EE[ Y(n)^6]$.
\item $f = f \otimes f \otimes f$ and $ \langle f \otimes f, f \rangle_{\EH^{\otimes 2}} =0$.
\item There exists $h^i \in \EH$ for $i = 1, \ldots, 2n$, such that $\|h^i \|_{\EH} =1$, $\langle h^i, h^j \rangle_{\EH} =0 $ for $i \not= j$ and
$$
I_2(f) = \sum_{i=1}^n I_2(h^i \otimes h^i) -  \sum_{i=n+1}^{2n} I_2(h^i \otimes h^i) = \sum_{i=1}^n (I_1(h^i)^2-1) -   \sum_{i=n+1}^{2n}  (I_1(h^i)^2 -1).
$$
\end{enumerate} 
\end{corollary}

\begin{proof}
It remains to prove the implication $(c) \Longrightarrow (d)$. If (c) is verified, then for every $j \geq 1$ we obtain $\lambda_{f,j} = \lambda_{f,j}^3$ and hence
$\lambda_{f,j} \in \{-1,+1\}$. Since $\sum_{j \geq 1} \lambda_{f,j}^2= 4n$ and $\sum_{j \geq 1} \lambda_{f,j}^3= 0$, we deduce that there are $2n$ indices $j$
with $\lambda_{f,j}=1$ and $2n$ indices with $\lambda_{f,j}=-1$.  The conclusion follows from \eqref{expansionHS}.
\end{proof}

\begin{remark} \rm
The statement of Corollary \ref{char} remains true for arbitrary parameters $\lambda>0$, not only for $\lambda=1/2$. The choice $\lambda=1/2$ just leads to the simple values $\lambda_{f,j} \in \{-1,+1\}$. In general we would obtain $\lambda_{f,j} \in \{-\frac{1}{2\lambda},\frac{1}{2\lambda}\}$.\\ On the other hand, suppose that $I_2(f)$ with $f \in \EH^{\odot 2}$ is such that $\EE[I_2(f)^2] = 8 r$ for some $r >0$. If $I_2(f)$ is distributed according to $\Gamma_s( \lambda, r)$,
then necessarily $2 r$ is an integer and $I_2(f)$ has a $\Gamma_s( \lambda, r)$-distribution. This follows immediately as in the proof of Corollary \ref{char}.
\end{remark}

\subsection{Homogeneous sums and multivariate extensions}\label{deJong}

Let $\bX=\{X_n\}_{n\in\NN}$ be a sequence of independent and identically distributed centred random variables with unit variance. Fix an integer $q\geq 2$ and let, for each $n\in\NN$, $h_n:\{1,\ldots,n\}^q\to\RR$ be a symmetric function, which vanishes on diagonals in the sense that $h_n(i_1,\ldots,i_q)=0$ whenever there are at least two indices $i_j\neq i_k\in\{1,\ldots,q\}$ such that $i_j=i_k$. Based on this data we define the sequence $\{H_n(\bX,q)\}_{n\in\NN}$ of homogeneous sum of order $q$ as
$$
H_n(\bX,q) := \sum_{1\leq i_1,\ldots,i_q\leq n}h_n(i_1,\ldots,i_q)\,X_{i_1}\cdots X_{i_q} = q!\sum_{1\leq i_1<\ldots<i_q\leq n}h_n(i_1,\ldots,i_q)\,X_{i_1}\cdots X_{i_q}.
$$
Universality for the family $\{H_n(\bX,q)\}_{n\in\NN}$ is a probabilistic phenomenon which asserts that $H_n(\bX,q)$ converges, as $n\to\infty$, to a limiting random variable $Y$ if and only if $H_n(\bG,q)$ converges in distribution to $Y$, where $\bG=\{G_n\}_{n\in\NN}$ is some particular sequence of independent and identically distribution random variables with mean zero and variance one. In our case, we take for $G_n$ a standard Gaussian random variable for each $n\in\NN$ (whence the notation $\bG$). Usually, it is much easier to show convergence in distribution of $H_n(\bG,q)$ than of $H_n(\bX,q)$. One reason for that being the interpretation of $H_n(\bG,q)$ as a multiple stochastic integral of order $q$, i.e., $H_n(\bG,q)=I_q(f_n)$ with $f_n\in{\EH}^{\odot q}$ given by 
$$
f_n = q!\sum_{1\leq i_1<\ldots<i_q\leq n}h_n(i_1,\ldots,i_q)\,e_{i_1}\otimes\cdots\otimes e_{i_1}.
$$
Moreover, a number of combinatorial tools are available to control the moments of such integrals, see \cite{PeccatiTaqqu:book} for details.

The universality phenomenon for homogeneous sums has been addressed by Rotar \cite{Rotar:1979} and later also by Nourdin, Peccati and Reinert \cite{NourdinPeccatiReinert:2010}, who consider especially multivariate extensions in case of normal and Gamma limiting distributions by means of Stein's method and Malliavin calculus. Using the results obtained in the previous sections, we can reduce a corresponding limit theorem to a simple moment condition in case $q=2$ and if the limiting distribution belongs to the broad class of Variance-Gamma distributions.

\begin{proposition}
Suppose that $\EE[H_n(\bG,q)^2]\to r(\sigma^2+2\theta^2)$, as $n\to\infty$, and let $Y$ be a random variable having a $VG_c(r,\theta,\sigma)$-distribution with parameters $r,\sigma>0$ and $\theta\in\RR$. Then, as $n\to\infty$, the following assertions are equivalent:
$$\text{(a) $H_n(\bX,q)$ converges in distribution to $Y$,}\qquad \text{(b) $H_n(\bG,q)$ converges in distribution to $Y$.}$$
If $q=2$ then (a) and (b) are equivalent to $\EE[H_n(\bG,2)^j]\to\EE[Y^j]$ for $j=3,4,5,6$. If $q=2$ and $\theta=0$ then (a) and (b) are even equivalent to $\EE[H_n(\bG,2)^4]\to\EE[Y^4]$ and $\EE[H_n(\bG,2)^6]\to\EE[Y^6]$.
\end{proposition}
\begin{proof}
The first part of the claim is a reformulation of a special case of Proposition 1 in \cite{Rotar:1979}. The second part is a direct consequence of Theorems \ref{thm:DoubleIntegralsRingTheta=0} and \ref{thm:DoubleIntegralsRingThetaNot0}.
\end{proof}

We now turn to a multivariate version of the results presented in Section \ref{sec:GeneralBound}. For this, fix $d\geq 2$ and let for each $n\in\NN$ and $j=1,\ldots,d$, $F_{n,j}\in\DD^{2,4}$ be such that $\EE[F_{n,j}]=0$. Let further $Y_j$ be $VG_c(r_j,\theta_j,\sigma_j)$-distributed with parameters $r_j,\sigma_j>0$ and $\theta_j\in\RR$ for all $j=1,\ldots,d$, form the sequence $\{\bF_n\}_{n\in\NN}$ of random vectors $\bF_n:=(F_{n,1},\ldots,F_{n,d})$ and put $\bY:=(Y_1,\ldots,Y_d)$. Next, define the sequence $\{A_n(j)\}_{n\in\NN}$ by
\begin{equation}\label{eq:defAnj}
A_n(j) := \EE\big[\big|\sigma_j^2(F_{n,j}+r_j\theta_j)-2\theta_j\Gamma_2(F_{n,j})-\Gamma_3(F_{n,j})\big|\big]+\big|r_j\sigma_j^2+2r_j\theta_j^2-\EE[\Gamma_2(F_{n,j})]\big|,
\end{equation}
and for $j\neq i=1,\ldots,d$ define $\{B_n(i,j)\}_{n\in\NN}$ by
$$
B_n(i,j) := \EE\big[\big|\lan DF_{n,i},-DL^{-1}F_{n,j} \ran_{\EH}\big|\big].
$$
A distance $d(\bF_n,\bY)$ between the random vectors $\bF_n$ and $\bY$ is measured by $$d(\bF_n,\bY):=\sup\big|\EE[\phi(\bF)]-\EE[\phi(\bY)]\big|,$$ where the supremum is taken over all functions $\phi:\RR^d\to\RR$ possessing partial derivatives of order one and two, which are uniformly bounded in absolute value by $1$. The distance $d(\,\cdot\,,\,\cdot\,)$ is our multivariate version of the Wasserstein distance used in the one-dimensional situation. The proof of the next result closely follows the lines of the proof of Lemma 4.4. in \cite{PeccatiThaele:2013}, which in turn was inspired by the methods in \cite{BourguinPeccati:Portmanteau}. To keep the paper reasonably self-contained, we decided yet to present the basic idea.

\begin{proposition}\label{pro:MultiRate}
There are constants $C_1>0$ and $C_2>0$ only depending on $d$ and the parameters $r_j$, $\theta_j$ and $\sigma_j$, $j=1,\ldots,d$, such that 
\begin{equation}\label{eq:MultiBound}
d(\bF_n,\bY)\leq C_1\sum_{j=1}^dA_n(j)+C_2\sum_{i,j=1\atop i\neq j}^d B_n(i,j).
\end{equation}
\end{proposition}
\begin{proof}
To simplify the notation and to keep the argument more transparent, we restrict to the case bivariate $d=2$. Thus, what we have to show is that
\begin{equation}\label{eq:MultiBoundd=2}
d(\bF_n,\bY)\leq C_1\big(A_n(1)+A_n(2)\big)+C_2\big(B_n(1,2)+B_n(2,1)\big),\qquad n\in\NN,
\end{equation}
with constants $C_1,C_2>0$ only depending on the parameters $r_1,r_2$, $\theta_1,\theta_2$ and $\sigma_1,\sigma_2$. We start by writing for an admissible test function $\phi:\RR^2\to\RR$,
\begin{align*}
\big|\EE[\phi(F_{n,1},F_{n,2})]-\EE[\phi(Y_1,Y_2)]\big|
&\leq \big|\EE[\phi(F_{n,1},F_{n,2})]-\EE[\phi(Y_1,F_{n,2})]\big|\\
&\qquad +\big|\EE[\phi(Y_1,F_{n,2})]-\EE[\phi(Y_1,Y_2)]\big|=:|T_1|+|T_2|.
\end{align*}
Conditioning on the first component $Y_1$ of $\bY$ in $T_2$ leads to a one-dimensional situation, which has already been considered in the proof of Theorem \ref{result1}. This contributes the term $A_n(2)$ to the bound \eqref{eq:MultiBoundd=2}. Let us turn to $T_1$. Writing $\cL_X$ for the distribution of an arbitrary random element $X$, we re-write $T_1$ as
$$
T_1=\int\Big(\phi(x,y)-\int\phi(z,y)\,\cL_{Y_1}(\dint z)\Big)\cL_{(F_{n,1},F_{n,2})}(\dint(x,y)).
$$
The term in brackets is now interpreted as the left-hand side of a Stein equation for $Y_1$, i.e.
\begin{equation}\label{eq:T1Expession}
T_1=\int \sigma_1^2(x+r_1\theta_1)h_y''(x)+(\sigma_1^2r_1+2(x+r_1\theta_1)h_y'(x)-xh_y(x)\,\cL_{(F_{n,1},F_{n,2})}(\dint(x,y)).
\end{equation}
Here, for fixed $y$, $h_y(x)$ stands for a solution of this equation for the text function $x\mapsto\phi(x,y)$. Also put $h(x,y):=h_y(x)$, understood as a bivariate function. Using the smoothness properties of the test function $\phi$ together with the smoothness properties of $h_y(x)$ (again taken from Lemma 3.17 in \cite{Gaunt:thesis}), we see that
\begin{itemize}
\item[(i)] the mappings $x\mapsto h(x,y)$ and $y\mapsto h(x,y)$ are twice differentiable on $\RR$,
\item[(ii)] there is a constant $C>0$ only depending on $r_1,\theta_1,\sigma_1$ such that all partial derivatives up to order two of the mappings in (i) are bounded by $C$
\end{itemize}
(compare with the proof of Lemma 4.4 in \cite{PeccatiThaele:2013} for a similar argument).
In terms of $h(x,y)$, the representation \eqref{eq:T1Expession} of $T_1$ can be re-written as
\begin{equation}\label{eq:T1}
\begin{split}
T_1 &=\EE\big[\sigma_1^2(F_{n,1}+r_1\theta_1)\partial_{xx}h(F_{n,1},F_{n,2})+(\sigma_1^2r_1+2\theta_1(F_{n,1}+r_1\theta_1))\partial_xh(F_{n,1},F_{n,2})\\
&\hspace{8cm}-F_{n,1}h(F_{n,1},F_{n,2})\big],
\end{split}
\end{equation}
where $\partial_x$ and $\partial_{xx}$ indicate the first and second partial derivative in the first coordinate (similarly, we write $\partial_y$ and $\partial_{yy}$ for those in the second coordinate). Applying the integration-by-parts-formula \eqref{intby} together with the chain rule \eqref{chainrule} we see that
\begin{align*}
&\EE[F_{n,1}h(F_{n,1},F_{n,2})]\\ 
&= \EE[\lan Dh(F_{n,1},F_{n,2}),-DL^{-1}F_{n,1}\ran_{\EH}]\\
&=\EE[\partial_xh(F_{n,1},F_{n,2})\lan DF_{n,1},-DL^{-1}F_{n,1}\ran_{\EH}+\partial_yh(F_{n,1},F_{n,2})\lan DF_{n,2},-DL^{-1}F_{n,1}\ran_{\EH}].
\end{align*}
Combining this with \eqref{eq:T1} and arguing as in the proof of Theorem \ref{result1}, we see that this contributes the terms $A_n(1)$ and $B_n(2,1)$ to \eqref{eq:MultiBoundd=2}. Interchanging the role of $F_{n,1}$ and $F_{n,2}$ leads to a term $B_n(1,2)$ and completes the argument.
\end{proof}

We now apply Proposition \ref{pro:MultiRate} to sequences of vectors belonging to a fixed Wiener chaos, i.e., we assume from now on that $F_{n,j}=I_{q_j}(f_{n,j})$ with $f_{n,j}\in\EH^{\odot q_j}$, where $q_1,\ldots,q_d\geq 2$. The next result ensures that convergence in distribution of the components of $\bF_n$ towards the components of $\bY$ already implies convergence in distribution of the involved random vector. This can be regarded as a quantitative version for Variance-Gamma distributions of the strong asymptotic independence properties on the Wiener chaos (see Remark \ref{rem:sai} below for further discussion).

\begin{proposition}\label{prop:MultiLimitTheorem}
Suppose that for each $j=1,\ldots,d$, $F_{n,j}$ converges in distribution to $Y_j$ and that for all $i\neq j=1,\ldots,d$, ${\rm Cov}(F_{n,i}^2,F_{n,j}^2)\to 0$, as $n\to\infty$. Then $\bF_n$ converges in distribution to $\bY$ and $$d(\bF_n,\bY)\leq C_1\sum_{j=1}^dA_n(j)+C_2\sum_{i,j=1\atop i\neq j}^d{\rm Cov}(F_{n,i}^2,F_{n,j}^2)$$ with $A_n(j)$ given by \eqref{eq:defAnj} and constants $C_1,C_2>0$ as in Proposition \ref{pro:MultiRate}.
\end{proposition}
\begin{proof}
In view of Proposition \ref{pro:MultiRate} and Theorem \ref{result1} it only remains to show that $B_n(i,j)$ is dominated by ${\rm Cov}(F_{n,i}^2,F_{n,j}^2)$ up to a constant factor. However, this is known from step 2 in the proof of \cite[Theorem 4.3]{NourdinRosinski:2014}, see also Identity (6.2.3) in \cite{NP:book}.
\end{proof}

\begin{remark}\label{rem:sai}\rm
Without a rate of convergence, Proposition \ref{prop:MultiLimitTheorem} is also a consequence of the strong asymptotic independence properties inside the Wiener chaos. In particular, the result is a consequence of Theorem 1.4 in \cite{NourdinNualartPeccati:2014} and the fact that the distribution of each $Y_j$, $j=1,\ldots,d$, is determined by its moments (alternatively, one can apply Theorem 3.1 in \cite{NourdinRosinski:2014}).
\end{remark}

\subsection*{Acknowledgements}
We are grateful to Ehsan Azmoodeh, Giovanni Peccati and Guillaume Poly for sharing their results with us.
We like to thank Robert Gaunt for helpful discussions on uniform bounds for the solutions of the Stein equation for nonsymmetric Variance-Gamma distributions.\\
The authors have been supported by the German research foundation (DFG) via SFB-TR 12. 


\newcommand{\SortNoop}[1]{}\def\cprime{$'$} \def\cprime{$'$}
  \def\polhk#1{\setbox0=\hbox{#1}{\ooalign{\hidewidth
  \lower1.5ex\hbox{`}\hidewidth\crcr\unhbox0}}}
\providecommand{\bysame}{\leavevmode\hbox to3em{\hrulefill}\thinspace}
\providecommand{\MR}{\relax\ifhmode\unskip\space\fi MR }
\providecommand{\MRhref}[2]{%
  \href{http://www.ams.org/mathscinet-getitem?mr=#1}{#2}
}
\providecommand{\href}[2]{#2}

\end{document}